\documentclass{myclass}
\usepackage{amsmath,amssymb,amsfonts,upref,endnotes}
\usepackage{amsthm}
\usepackage{marginnote}
\usepackage{soul}
\usepackage[nospace]{cite}
\newcommand{\absbreak}{}

\newtheorem{theorem}{\bf Theorem}[section]
\newtheorem{proposition}{\bf Proposition}[section]

\newtheorem{remark}{\bf Remark}[section]
\usepackage{mathrsfs}
\usepackage{bm}
\usepackage{stmaryrd}

\newcommand{\Act}{\mathscr{A}}
\newcommand{\Lag}{\mathscr{L}}
\newcommand{\Ham}{\mathcal{H}}

\newcommand{\Xiset}{\Xi}
\newcommand{\Bset}{\mathscr{B}}

\newcommand{\Dt}{\mathrm{D}_t}
\newcommand{\Id}{\mathbf{I}_d}
\newcommand{\Traj}{\bm{\Phi}}
\newcommand{\txint}{\int_t\!\int_{\Omega_t}}
\newcommand{\half}{\tfrac{1}{2}}

\newcommand{\mtrx}[1]{\mathbf{#1}}

\newcommand{\vel}{\bm{u}}
\newcommand{\lift}{\bm{f}}

\newcommand{\pth}[1]  {\left( #1 \right)}
\newcommand{\brck}[1]  {\left[ #1 \right]}
\newcommand{\acco}[1]{\left\{ #1 \right\}}

\newcommand{\rot}{\mathbf{rot}}
\newcommand{\jump}[1]{\llbracket #1\rrbracket} 
\begin{document}
\title{An all-topology two-fluid model for two-phase flows derived through Hamilton's Stationary Action Principle}
\author{Ward Haegeman$^{1,2}$,
Giuseppe Orlando$^{2}$,
Samuel Kokh$^{3}$
and
Marc Massot$^{2}$}

\address{$^{1}$DMPE, ONERA, Université Paris-Saclay, F-91123 Palaiseau, France\\
$^{2}$CMAP, CNRS, École polytechnique, Institut Polytechnique de Paris, Route de Saclay, 91120 Palaiseau, France\\
$^{3}$Maison de la Simulation, UAR 3441, Bâtiment 565 PC 190 - Digiteo, CEA Saclay 91191, Gif-sur-Yvette, France}

\subject{applied mathematics, mathematical modelling, fluid mechanics}


\corres{Ward Haegeman\\
\email{ward.haegeman@polytechnique.edu}} 

\maketitle
\begin{abstract}
We present a novel multi-fluid model for compressible two-phase flows. The model is derived through a newly developed Stationary Action Principle framework. It is fully closed and introduces a new interfacial quantity, the interfacial work. The closures for the interfacial quantities are provided by the variational principle. They are physically sound and well-defined for all types of flow topologies. The model is shown to be hyperbolic, symmetrizable, and admits an entropy conservation law. Its non-conservative products yield uniquely defined jump conditions which are provided. As such, it allows for the proper treatment of weak solutions. In the multi-dimensional setting, the model presents lift forces which are discussed. The model constitutes a sound basis for future numerical simulations. \absbreak
\end{abstract} 
\newpage

\section{Introduction}

The class of multi-fluid models is widely accepted among the modelling strategies to accurately describe two-phase flows. Moreover, multi-fluid models are viable in a wide range of flow regimes and can offer practical advantages. For instance, multi-fluid models, obtained through the method of moments, allow for an Eulerian description of disperse flows that may include poly-dispersion and small-scale dynamics \cite{loison:2025}. They are a cost-effective alternative to the usual Lagrangian models for which the sampling cost increases drastically with the dimension of the phase space.
In the context of separated phase flows, multi-fluid models avoid the need for any explicit interface tracking or reconstruction by relying on a diffuse interface representation, and therefore, they are naturally adapted to handle important topology changes. Moreover, recent contributions have proposed two-scale multi-fluid models to describe both the separated and disperse regimes \cite{frederix:2021,loison:2024}. Yet, there is still no universally accepted basis multi-fluid formulation for two-phase flows and the development of new models remains an active research area. In particular, the development of models that are out of velocity equilibrium remains challenging, yet necessary to accurately describe phenomena such as atomization, slip-flow, added-mass, particle inertia, \textit{etc}. One of the key challenges is the derivation of a model that is well-posed and allows for the correct treatment of shocks, while being physically valid over a sufficiently wide range of flow regimes.

Among the existing two-velocity models, single-pressure models are generally non-hyperbolic \cite{ishii:1975,stuhmiller:1977,ramshaw:1978,vazquez-gonzalez:2016}. Although a hyperbolic single-pressure model has been recently derived through kinetic theory in the context of disperse flows \cite{fox:2020}, the model lacks the mathematical foundations to properly treat shocks as its jump conditions are not well-defined. Moreover, it lacks a supplementary conservation law for the entropy. Consequently, as an alternative to single-pressure models, many efforts have been devoted to the development of two-pressure models \cite{ransom:1984,baer:1986,chen:1996,saurel:1999}. Among these models, the Baer-Nunziato model \cite{baer:1986}, initially proposed for deflagration-to-detonation transition in granular materials, has been extended into a class of two-phase flow models which has become quite popular \cite{saurel:1999,saurel:2001,gallouet:2004,tokareva:2010,coquel:2013,hantke:2019,thanh:2022}. Not relying on any equilibrium assumption, the model's dependent variables are the phasic flow variables and the volume fraction. The volume fraction advection term, as well as the momentum and energy exchange terms, depend on interfacial quantities, namely the interfacial velocity $\vel_I$, and the interfacial pressure $p_I$, which must be closed as functions of the dependent variables. The closure problem is common to many two-pressure models and several expressions for the interfacial quantities have been proposed \cite{ransom:1984,chen:1996,saurel:1999,saurel:2001,saurel:2003,gallouet:2004,jin:2006,perrier:2021}. A numerical comparison between some of the suggested closures can be found in \cite{guillemaud:2007,lochon:2016}. Among the different possibilities, the original formulation \cite{baer:1986} suggests that the interfacial velocity be taken equal to the velocity of one phase, generally the least compressible one, while the interfacial pressure be taken as the pressure of the other phase. 

Although practical, such simple closures are adapted to specific types of flows and may not be sufficiently general to treat complex flows, which may involve important topology changes and even phase inversions. Other closures for the interfacial quantities have also been proposed. For instance for the velocity, the volume-weighted average $\vel_I=\alpha_1 \vel_1 + \alpha_2 \vel_2$, or the mass-weighted average $\vel_I = Y_1 \vel_1 + Y_2 \vel_2$, where $\vel_k$ is the velocity of the $k^\text{th}$ phase, while $\alpha_k$ and $Y_k$ are its volume fraction and mass fraction respectively. Although for specific regimes, these may not be the most accurate estimations of the true interfacial velocity, they present the advantage of being well-defined over the whole domain regardless of the flow configuration. Therefore, we refer to models with such interfacial quantities as \textit{all-topology} models.

Another aspect to consider when closing the interfacial quantities is their impact on the mathematical structure of the resulting model. When the interfacial quantities are functions of only the dependent variables --- and not their gradients as in \cite{saurel:2003,perrier:2021}--- the model is hyperbolic under some non-resonance condition. However, as these interfacial quantities define the inter-phase momentum and energy exchanges, they appear in the governing equations as non-conservative products. As a consequence, jump conditions for shocks may not be well-defined and different numerical schemes may converge to different solutions, as illustrated in \cite{guillemaud:2007}. For the Riemann problem to be well-posed, first, the characteristic field associated to the interfacial velocity should be linearly degenerate in order for the jump conditions to be uniquely defined. Second, the system's total entropy equation should satisfy a conservation law in order to select admissible weak solutions. The closures which satisfy these conditions for the Baer-Nunziato model have been identified in \cite{gallouet:2004}. In order to obtain a linearly degenerate characteristic field, the interfacial velocity must be chosen as either one of the phasic velocities or as the mass-weighted average velocity.
Using the mass-weighted average velocity as interfacial velocity has been proposed in several contributions, in combination with an interfacial pressure that is usually given by the mixture pressure $p_I = \alpha_1 p_1 + \alpha_2 p_2$ \cite{saurel:1999}, or a mixture total pressure $p_I = \alpha_1 p_1 + \alpha_2 p_2 + \rho Y_1 Y_2 (\vel_1 - \vel_2)^2$ \cite{saurel:2001} --- where $\rho=\alpha_1 \rho_1 + \alpha_2 \rho_2$ is the total density and $\rho_k$, $k=1,2$, are the phasic densities --- which is the pressure Riemann invariant and therefore continuous through interfaces. However, these expressions for the interfacial pressure do not satisfy the second requirement for a well-posed Riemann problem, since they do not lead to a conservation law for the system's entropy. As it turns out, when the interfacial velocity is the mass-weighted average velocity, the only expression for the interfacial pressure that leads to a conservative entropy equation is given by
\begin{equation}
p_I = \frac{Y_2 T_2 p_1 + Y_1 T_1 p_2}{Y_2 T_2 + Y_1 T_1},
\label{eq:p_I_BN}
\end{equation}
where $p_k$ and $T_k$ are respectively the pressure and temperature of the $k^\text{th}$ phase \cite{gallouet:2004}. However, since the interfacial pressure describes a momentum exchange due to an imbalance of forces at the interface, it is related to mechanical effects rather than thermal ones. As such, it is not expected to depend on the phasic temperatures, and to the best of our knowledge, there is no clear physical interpretation for expression \eqref{eq:p_I_BN}, which is obtained through purely mathematical considerations by enforcing an entropy conservation law. As such, the class of Baer-Nunziato models does not provide an \textit{all-topology} model which has both the required mathematical properties to treat shocks, and physical consistency.

An alternative \textit{all-topology} compressible two-fluid description \cite{romenski:2004, romenski:2007, romenski:2010} has been derived by means of the theory of thermodynamically compatible systems \cite{godunov:1961, godunov:1996, romensky:1998, ferrari:2025}. In its barotropic version, the model has been extensively studied \cite{thein:2022a,thein:2022b}, while an alternative derivation, and its extension to include velocity fluctuations, through Stationary Action Principle can be found in \cite{haegeman_RelVel:2025}. In the mono-dimensional setting, the barotropic model is equivalent to a barotropic Baer-Nunziato model with closures $\vel_I = Y_1\vel_1 +Y_2\vel_2$ and $p_I = Y_2 p_1 + Y_1 p_2$, but expressed in terms of total momentum, $\rho \vel = \alpha_1\rho_1 \vel_1 + \alpha_2\rho_2\vel_2$ --- notice that here $\vel=\vel_I$ --- and the relative velocity, $\bm{W}=\vel_1-\vel_2$, instead of the phasic momenta $\alpha_k\rho_k\vel_k$, $k=1,2$. In the multi-dimensional setting, compared to the Baer-Nunziato model, additional lift-like forces appear, as well as the involution constraint $\rot(\bm{W})=\bm{0}$. However, as discussed in \cite{haegeman_RelVel:2025}, the main challenge related to this model is its extension to full thermodynamics, i.e. to non-barotropic flows. Several versions of the model which include full thermodynamics have been proposed, as either single- or two-temperature models \cite{romenski:2004,romenski:2007,romenski:2010}, but all share the same type of underlying assumption, which is the conservation of entropy with mixture velocity $\vel$. For smooth solutions of the Euler equations, the entropy density is conserved by the flow. Extending this to two-fluid models, one may assume the conservation of the phasic entropy density, $\alpha_k \rho_k s_k$, $k=1,2$, following its own velocity $\vel_k$. This then leads to the following equation for the total entropy density, $\rho s = \alpha_1 \rho_1 s_1 + \alpha_2 \rho_2 s_2$, written in terms of mixture velocity $\vel$ and relative velocity $\bm{W}$,
\begin{equation}
\partial_t\pth{\rho s} + \nabla\cdot\pth{\rho s \vel + \rho Y_1 Y_2 (s_1 - s_2)\bm{W}} = 0.
\label{eq:entropy_eq_u_W}
\end{equation}
In the framework of thermodynamically compatible systems, phasic entropies are assumed to be conserved in the frame associated to the mixture velocity instead of their own velocity. This leads to a total entropy equation for which the entropy flux proportional to $\bm{W}$ in \eqref{eq:entropy_eq_u_W} is absent \cite{romenski:2004,romenski:2007,romenski:2010,laspina:2017}. Consequently, non-zero relative velocity induces an entropy exchange between the phases which induces an artificial heat flux in the total energy equation, which then writes 
\begin{equation}
\partial_t\pth{\alpha_1\rho_1 E_1 + \alpha_2 \rho_2 E_2} + \nabla\cdot\pth{\alpha_1\rho_1 H_1 \vel_1 + \alpha_2 \rho_2 H_2 \vel_2 - \rho Y_1 Y_2 (s_1 T_1 - s_2 T_2)\bm{W}} = 0,\label{eq:artificial_heat_SHTC}
\end{equation}
where $E_k=e_k+\vel_k^2/2$, and $H_k=h_k+\vel_k^2/2$, are respectively the specific total energy and specific total enthalpy of the $k^\text{th}$ phase, with $e_k$, and $h_k=e_k + p_k/\rho_k$, its specific internal energy and specific enthalpy. Note that no entropy production is associated to this heat flux term and that it is not related to any heat diffusion phenomenon. Moreover, the framework then also leads to the appearance of non-conservative products proportional to the phasic temperature gradients in the phasic momentum equations and for which there is no clear physical interpretation. As a result, the quest for an \textit{all-topology} two-fluid model, which has both the mathematical properties required to ensure its well-posedness and to properly treat shocks, as well as the physical consistency that allows for a clear interpretation of its various terms, does not seem to have received a conclusive answer yet. 

In this contribution, we propose a novel \textit{all-topology} two-fluid model derived through Hamilton's Stationary Action Principle. Compared to the variational principle used in \cite{haegeman_RelVel:2025}, which relies on a single family of trajectories, here the variational principle relies on two families of trajectories, one for each phase. Whereas previous contributions using two families of trajectories, derived single-pressure models \cite{bedford:1978,geurst:1986}, or models with the interfacial velocity equal to the velocity of one of the phases \cite{gavrilyuk:2002}, here we propose a novel derivation of an original two-pressure model which considers the mixture velocity as interfacial velocity. The resulting model is fully closed and the resulting expressions for the interfacial quantities recover previously used and physically sound expressions. Moreover, a new interfacial quantity, representing the inter-phase work produced at the interface, appears, and is coined interfacial work accordingly. The model is analysed in the one-dimensional setting, under a non-resonance condition --- which is classical for such two-fluid models \cite{ransom:1984,gallouet:2004} --- the model is hyperbolic and symmetrizable, which ensures its local well-posedness. Furthermore, owing to the linearly degenerate structure of the interfacial wave, its non-conservative products are well-defined for weak solutions and thus the model admits a unique set of jump conditions. The conservation law satisfied by the total entropy allows to select admissible shocks in agreement with the second principle of thermodynamics. As such, the model we propose connects with both the mathematical modelling approach, which relies on PDE analysis tools, and the mechanical modelling approach, which seeks clear physical interpretations of the various terms.

In the multi-dimensional setting, an additional force, akin to lift, appears in the phasic momentum equations and is a direct consequence of the imposed interfacial velocity in the Hamiltonian Principle. Its expression implies the relative velocity between the two phases but also a new potential velocity field which has its own evolution equation, and is generated by the pressure non-equilibrium between the two phases. The presence of lift forces has also been obtained in the thermodynamically compatible two-phase models \cite{romenski:2004, romenski:2007, romenski:2010}, and in the Hamiltonian models found in \cite{gavrilyuk:2020}, but are not yet well understood. We compare their expression in both formulations but the proper modelling of such effects requires further studies.

The outline of this contribution is as follows. We end this section by specifying some notations and vector analysis conventions. Next, in Section \ref{sec:SAP_framework}, the two-trajectory Stationary Action Principle framework is presented. Then, in Section \ref{sec:Model_derivation}, we derive and discuss the new model. Its mathematical analysis in the one-dimensional setting is conducted in Section \ref{sec:Model_analysis}, some comments on the multi-dimensional setting are also provided. Final remarks and conclusions are presented in Section \ref{sec:Conclusions}. Appendices \ref{app:SAP_proof} and \ref{app:derivation_eulerienne} provide supplementary details concerning the modelling framework.\\

We adopt the following conventions. Vectors of $\mathbb{R}^d$ are represented in column form, and the identity matrix is denoted $\Id$. The derivative of a scalar field with respect to a vector field is its differential and is represented by a row-vector, its gradient is the transpose of its differential and is a column-vector. Similarly, the derivative of a vector with respect to another vector is its Jacobian matrix. The gradient of a vector field is defined as the transpose of its Jacobian matrix. The divergence of a matrix is the column vector obtained by applying the divergence operator to each row-vector of the matrix.

For a phasic velocity field $\vel_k$, the associated material derivative is denoted $\Dt^k$, i.e. $\Dt^k\bullet = \partial_t\bullet + (\vel_k\cdot\nabla)\bullet$. The mass-weighted averaged velocity $Y_1\vel_1 + Y_2 \vel_2$ is denoted $\vel$ as indicated previously, it will also be referred to as mixture velocity. The material derivative operator defined by this velocity is simply denoted $\Dt$. Interfacial velocities are denoted $\vel_I$, when the interfacial velocity equals the mixture velocity, its index may be dropped without ambiguity.

\section{The Stationary Action Principle with two families of trajectories}\label{sec:SAP_framework}

Hamilton's Stationary Action Principle allows for a variational derivation of equations of motion under the reversibility assumption. Its application to fluid mechanics allows to derive the equations for ideal fluids in both Lagrangian and Eulerian coordinates \cite{bretherton:1970,salmon:1988,serre:1993}.
The Lagrangian derivations are based on the introduction of fluid particle trajectories $\bm{x}=\Traj(t,\bm{X})$, where $\bm{x}$ and $\bm{X}$ are the Eulerian and Lagrangian coordinates respectively. The perturbations of the different flow variables are then related to perturbations of the trajectories and naturally account for constraints such as mass conservation. 

A single-velocity generic Stationary Action Principle framework has been presented in \cite{haegeman_GenSAP:2025} and used for the derivation of the hyperbolic model for heat transfer presented in \cite{dhaouadi:2024}. Extension to models without velocity equilibrium has been achieved in \cite{haegeman_RelVel:2025} by assuming the single trajectory to be a virtual trajectory associated to the mass-weighted average velocity and by adding the relative velocity as an internal variable associated to phasic mass conservation constraints. This has allowed for an original derivation of the barotropic version of the two-phase flow model originally obtained through the theory of thermodynamically compatible systems in \cite{romenski:2004, romenski:2007, romenski:2010}, as well as its extension to include pseudo-turbulence. However, as mentioned in the introduction, the lifting of the barotropic assumption in this single-trajectory framework leads to artificial heat exchanges.

The introduction of two families of trajectories, $\Traj_1$ and $\Traj_2$, has allowed for the application of the variational principle to fluid mixtures in the context of two-phase flows. Single-pressure models, classically obtained through phase averaging, have been obtained by application of the Stationary Action Principle to averaged energy densities \cite{bedford:1978}. Interactions between the velocity fields such as added-mass effects have also been considered \cite{geurst:1986}. Two-pressure models have also been derived \cite{gavrilyuk:2002} using one of the phasic velocities to govern the volume fraction dynamics. 

In this section, we extend the generic framework presented in \cite{haegeman_GenSAP:2025} into a two-trajectories framework. To do so, we consider a generic Lagrangian density $\Lag$, which is assumed of the following form,
\begin{equation}\label{eq:generic_Lagrangian}
\Lag = \Lag_1\pth{m_1,\vel_1,b_{1,i},b_{2,i}, \xi_j,\nabla\xi_j,\Dt^1\xi_j} +  \Lag_2\pth{m_2,\vel_2,b_{2,i},b_{1,i}, \xi_j,\nabla\xi_j,\Dt^2\xi_j}.
\end{equation}
Here, for $k=1,2$, $\vel_k$ denotes the velocity associated to the trajectory $\Traj_k$, while $m_k$ is the partial mass that is conserved along this trajectory. Variables $b_{k,i}$, $i=1,\ldots,n_k$ are transported at velocity $\vel_k$, they allow to represent constraints, which may be expressed as transport equations. As such, $b_{k,i}$ may play the role of a phasic entropy, mass fractions in the context of multi-component phases, turbulent entropies \cite{saurel:2003,herard:2003,haegeman_RelVel:2025}, \textit{etc}.  Finally, variables $\xi_j$, $j=1\ldots,m$, are \textit{free} variables in the sense that they are not related to any of the trajectories as they will have their own independent variations. These may represent coupling variables, sub-scale phenomena such as micro-inertia \cite{gavrilyuk:2002}, or be related to internal constraints.

The Stationary Action Principle then dictates that the equations of motion are critical points of the Action $\Act$, which is the space-time integral of the Lagrangian density $\Lag$. As such, the generic equations are determined by
\begin{equation}\label{eq:stationary_action}
\delta\Act = \delta\txint \Lag \mathrm{d}\bm{x}\mathrm{d}t = 0,
\end{equation}
subjected to the following constraints, related to the conservation of the partial masses,
\begin{subequations}\label{eq:generic_system_constraints}
\begin{align}
\partial_t m_1 + \nabla\cdot\pth{m_1\vel_1} &= 0,\\
\partial_t m_2 + \nabla\cdot\pth{m_2\vel_2} &= 0,
\end{align}
and the transport of the advected quantities,
\begin{align}
\Dt^1 b_{1,i} &= 0, \qquad i=1,\ldots,n_1,\\
\Dt^2 b_{2,i} &= 0, \qquad i=1,\ldots,n_2.
\end{align}
\end{subequations}
The space of perturbations is then parametrized by two one-parameter families of perturbed trajectories $\Traj_k(t,\bm{X},\varepsilon)$. These then define virtual displacements, $\delta\Traj_k(t,\bm{x})$, that we express in Eulerian coordinates. Perturbations of the velocities $\delta\vel_k$, partial masses $\delta m_k$, and advected quantities $\delta b_{k,i}$ are then related to the virtual motions  $\delta\Traj_k$, through the following relations,
\begin{subequations}\label{eq:SAP_variations}
\begin{align}
\delta\vel_k &= \Dt^k\pth{\delta\Traj_k} - \pth{\nabla\vel_k}^\top\delta\Traj_k, \\
\delta m_k &= -\nabla\cdot\pth{m_k\delta\Traj_k},\\
\delta b_{k,i} &= -\pth{\delta\Traj_k}^\top\nabla b_{k,i}.
\end{align}
These perturbations encode the constraints \eqref{eq:generic_system_constraints}, and their expressions are classical, see for instance \cite{bretherton:1970,gavrilyuk:2002,dhaouadi:2024,haegeman_GenSAP:2025}.

Variables $\xi_j$ are assumed to have independent variations $\delta\xi_j$. Commuting the $\delta$ operator with space and time derivatives, then yields the following expressions for the remaining variations,
\begin{align}
\delta\pth{\nabla\xi_j} &= \nabla\pth{\delta\xi_j},\\
\delta\pth{\Dt^k\xi_j} &= \Dt^k\pth{\delta\xi_j} + \pth{\nabla\xi_j}^\top\delta\vel_k.
\end{align}
\end{subequations}
Finally, following the Hamiltonian principle according to which all boundary points are kept fixed, the perturbations are assumed to be with compact support in the space-time domain. We then obtain the following results, proofs of which are reported in appendix \ref{app:SAP_proof}.

\begin{theorem}[The generic system]\label{thrm:SAP}
The Stationary Action Principle \eqref{eq:stationary_action}, applied to the generic Lagrangian \eqref{eq:generic_Lagrangian}, with variations \eqref{eq:SAP_variations}, corresponding to the constraints \eqref{eq:generic_system_constraints}, yields the following evolution equations,
\begin{subequations}\label{eq:generic_system_Euler_Lagrange}
\begin{equation}\label{eq:generic_system_K1_eq}
\begin{split}
\partial_t \bm{K}_1 + \nabla\cdot\pth{\bm{K}_1\otimes\vel_1 + \bm{\Pi}_1} =&  -\sum_{j=1}^{m} \brck{\partial_t\mathscr{M}_{1,j}+\nabla\cdot\pth{\mathscr{M}_{1,j}\vel_1} + \nabla\cdot\bm{D}_{1,j} - \dfrac{\partial\Lag_1}{\partial\xi_j}}\nabla\xi_j\\
&+ \sum_{i=1}^{n_2}\dfrac{\partial\Lag_1}{\partial b_{2,i}}\nabla b_{2,i} - \sum_{i=1}^{n_1}\dfrac{\partial\Lag_2}{\partial b_{1,i}}\nabla b_{1,i} ,
\end{split}
\end{equation}
\begin{equation}\label{eq:generic_system_K2_eq}
\begin{split}
\partial_t \bm{K}_2 + \nabla\cdot\pth{\bm{K}_2\otimes\vel_2 + \bm{\Pi}_2} =&  -\sum_{j=1}^{m} \brck{\partial_t\mathscr{M}_{2,j}+\nabla\cdot\pth{\mathscr{M}_{2,j}\vel_2} + \nabla\cdot\bm{D}_{2,j} - \dfrac{\partial\Lag_2}{\partial\xi_j}}\nabla\xi_j\\
&- \sum_{i=1}^{n_2}\dfrac{\partial\Lag_1}{\partial b_{2,i}}\nabla b_{2,i} + \sum_{i=1}^{n_1}\dfrac{\partial\Lag_2}{\partial b_{1,i}}\nabla b_{1,i} ,
\end{split}
\end{equation}
and
\begin{equation}\label{eq:generic_system_xi_eq}
\partial_t \pth{\mathscr{M}_{1,j} + \mathscr{M}_{2,j}} + \nabla\cdot\pth{\mathscr{M}_{1,j}\vel_1 + \mathscr{M}_{2,j}\vel_2 + \bm{D}_{1,j}+\bm{D}_{2,j}}  - \frac{\partial\Lag_1}{\partial\xi_j} - \frac{\partial\Lag_2}{\partial\xi_j}= 0,
\end{equation}
\end{subequations}
for $j=1,\ldots,m$. Here, $\bm{K}_1$ and $\bm{K}_2$ are the generalized partial momenta, while $\bm{\Pi}_1$ and $\bm{\Pi}_2$ are their associated generalized pressure tensors, they are defined as
\begin{subequations}
\begin{equation}
\bm{K}_k = \pth{\dfrac{\partial\Lag_k}{\partial\vel_k}}^\top, \qquad \bm{\Pi}_k=\pth{\Lag_k- m_k\dfrac{\partial\Lag_k}{\partial m_k}}\Id - \sum_{j=1}^m\nabla\xi_j\otimes\bm{D}_{k,j} ,\qquad k=1,2,
\end{equation}  
The Euler-Lagrange equations related to the free variables are expressed through their conjugate variables,
\begin{equation}
\mathscr{M}_{k,j} = \dfrac{\partial\Lag_k}{\partial\pth{\Dt^k\xi_j}}, \qquad \bm{D}_{k,j} = \pth{\dfrac{\partial\Lag_k}{\partial\pth{\nabla\xi_j}}}^\top,\qquad j=1,\ldots,m,\quad k=1,2.
\end{equation}
\end{subequations}
\end{theorem}

\begin{remark}
The right-hand sides of eq. \eqref{eq:generic_system_K1_eq} and eq. \eqref{eq:generic_system_K2_eq} represent inter-phase momentum exchanges. These terms are in general not zero. However, the generic system ensures the conservation of the total momentum, $\bm{K}_1+\bm{K}_2$, in agreement with Noether's theorem and thus  when summing eqs. \eqref{eq:generic_system_K1_eq} and \eqref{eq:generic_system_K2_eq}, their right-hand sides cancel each other out owing to eq. \eqref{eq:generic_system_xi_eq}.
\end{remark}

\begin{remark}\label{rmrk:constraints}
When $\Lag_1$ and $\Lag_2$ do not depend on $\Dt^1\xi_j$ and $\Dt^2\xi_j$, eq. \eqref{eq:generic_system_xi_eq} yields an algebraic constraint instead of an evolution equation.
\end{remark}

For convenience, in the following, we will denote by $\Bset_1=\{b_{1,i},\;i=1,\ldots,n_1\}$ and $\Bset_2=\{b_{2,i},\;i=1,\ldots,n_2\}$, the sets of transported variables, while $\Xiset =\{\xi_j,\;j=1,\ldots,m\}$ will denote the set of free variables.

As the generic Lagrangian \eqref{eq:generic_Lagrangian} does not explicitly depend on the time variable, the system's total energy, also called Hamiltonian, is conserved as specified by the following result.

\begin{theorem}[Conservation of the Hamiltonian]\label{thrm:Hamiltonian}
For $k=1,2$, define the partial Hamiltonian $\Ham_k$ as follows,
\begin{equation}\label{eq:generic_partial_hamiltonian_def}
\Ham_k = \bm{K}_k^\top\vel_k + \sum_{j=1}^m \mathscr{M}_{k,j}\Dt^k\xi_j - \Lag_k.
\end{equation}
The generic system composed of eqs. \eqref{eq:generic_system_constraints} and eqs. \eqref {eq:generic_system_Euler_Lagrange} admits a supplementary conservation law satisfied by the Hamiltonian, $\Ham=\Ham_1 + \Ham_2$, which writes
\begin{equation}\label{eq:generic_hamiltonian_eq}
\partial_t\Ham + \nabla\cdot\pth{\Ham_1\vel_1 +  \Ham_2\vel_2 + \bm{\Pi}_1^\top\vel_1 +\bm{\Pi}_2^\top\vel_2 + \sum_{j=1}^m\pth{\bm{D}_{1,j}\Dt^1\xi_j+\bm{D}_{2,j}\Dt^2\xi_j}} = 0.
\end{equation}
The equation for the partial Hamiltonian $\Ham_k$ writes
\begin{multline}\label{eq:generic_partial_hamiltonian_eq}
\partial_t\Ham_k + \nabla\cdot\pth{\Ham_k\vel_k + \bm{\Pi}_k^\top\vel_k+\sum_{j=1}^m\bm{D}_{k,j}\Dt^k\xi_j} = \sum_{i=1}^{n_k}\dfrac{\partial\Lag_{k'}}{\partial b_{k,i}}\partial_t b_{k,i} - \sum_{i=1}^{n_{k'}}\dfrac{\partial\Lag_{k}}{\partial b_{k',i}}\partial_t b_{k',i} \\
+\sum_{j=1}^{m} \brck{\partial_t\mathscr{M}_{k,j}+\nabla\cdot\pth{\mathscr{M}_{k,j}\vel_k} + \nabla\cdot\bm{D}_{k,j} - \dfrac{\partial\Lag_k}{\partial\xi_j}}\partial_t\xi_j,
\end{multline}
with $k'=3-k$, and $k=1,2$.
\end{theorem}

We conclude this section by showing how the present framework may be applied to derive the Baer-Nunziato two-phase flow model with its original closures  \cite{baer:1986}. We consider two compressible and immiscible fluids, each equipped with a complete equation of state, given by a strictly concave function $(\tau_k, e_k)\mapsto s_k$, with $\tau_k=1/\rho_k$, the specific volume. The Gibbs identity, $T_k\mathrm{d}s_k=\mathrm{d}e_k + p_k\mathrm{d}\tau_k$, then defines the phasic temperature, $T_k>0$, and the phasic pressure, $p_k$. The $k^\text{th}$ phase's occupancy rate is given by its volume fraction $\alpha_k\in(0,1)$, which satisfies the saturation constraint $\alpha_1+\alpha_2=1$. The Lagrangian is given by the difference between the kinetic energy and the potential energy which, in the absence of external forces reduces to the internal thermodynamic energy. Hence, we obtain
\begin{equation}\label{eq:two-phase_Lagrangian}
\Lag = \half m_1 \vel_1^2 - m_1 e_1\pth{\dfrac{m_1}{\alpha_1},s_1} + \half m_2 \vel_2^2 - m_2 e_2\pth{\dfrac{m_2}{1-\alpha_1},s_2},
\end{equation}
where $m_k=\alpha_k\rho_k$ is the partial mass of the $k^\text{th}$ phase. In the absence of heat exchange and under the reversibility assumption, fluid particles do not exchange entropy, nor is there any entropy creation, as such each fluid particle is assumed to carry a fixed amount of entropy. Moreover, following the original closures, we impose the interfacial velocity $\vel_I=\vel_1$ for the volume fraction transport. As a result, we consider the Lagrangian \eqref {eq:two-phase_Lagrangian} with $\Bset_1 = \{\alpha_1, s_1\}$, $\Bset_2 = \{s_2\}$, and $\Xiset = \emptyset$. Application of Theorem \ref{thrm:SAP} then yields the phasic momentum equations, while application of Theorem \ref{thrm:Hamiltonian} allows to exchange the phasic entropy equations for the phasic total energy equations as $\Ham_k=\alpha_k\rho_k E_k$. The complete system that results then writes
\begin{subequations}
\begin{align}
\partial_t\alpha_1 + \vel_1\cdot\nabla\alpha_1 &=0, \label{eq:BN_vol_frac_transport}\\
\partial_t\pth{\alpha_1 \rho_1} + \nabla\cdot\pth{\alpha_1 \rho_1 \vel_1} &=0,\\
\partial_t\pth{\alpha_2 \rho_2} + \nabla\cdot\pth{\alpha_2 \rho_2 \vel_2} &=0,\\
\partial_t\pth{\alpha_1\rho_1\vel_1} + \nabla\cdot\pth{\alpha_1\rho_1\vel_1\otimes\vel_1 + \alpha_1 p_1\Id} - p_2\nabla\alpha_1 &=0,\\
\partial_t\pth{\alpha_2\rho_2\vel_2} + \nabla\cdot\pth{\alpha_2\rho_2\vel_2\otimes\vel_2 + \alpha_2 p_2\Id} + p_2\nabla\alpha_1 &=0,\\
\partial_t\pth{\alpha_1\rho_1 E_1} + \nabla\cdot\pth{\alpha_1\rho_1 H_1\vel_1} - p_2\vel_1\cdot\nabla\alpha_1 &=0,\\
\partial_t\pth{\alpha_2\rho_2 E_2} + \nabla\cdot\pth{\alpha_2\rho_2 H_2\vel_2} + p_2\vel_1\cdot\nabla\alpha_1 &=0.
\end{align}
\end{subequations}
This is exactly the Baer-Nunziato system \cite{baer:1986} with interfacial velocity $\vel_I=\vel_1$ and interfacial pressure $p_I=p_2$. It is important to note that here, only the choice of $\vel_1$ as interfacial velocity is imposed as a constraint, since $\alpha_1\in\Bset_1$, however, the corresponding closure for the interfacial pressure results directly from the variational principle.

Much of the physical information about the system is contained within the Lagrangian, yet, the constraints associated to the variables play an equally important role. For instance, if we change the volume fraction constraint and set it to be a free variable, i.e. $\alpha_1 \in\Xiset$, instead of eq. \eqref{eq:BN_vol_frac_transport}, we obtain an algebraic constraint in accordance with Remark \ref{rmrk:constraints}. This constraint reads $p_1=p_2$, and results in a single-pressure model that is non-hyperbolic.

We now move on to the next section in which an all-topology two-fluid model is derived.

\section{An all-topology two-phase flow model}\label{sec:Model_derivation}

\subsection{Derivation of the model}

We consider the same physical setting as used in the previous section to derive the Baer-Nunziato model through Stationary Action Principle. In order to derive an all-topology two-fluid model, the constraint on the volume fraction must be modified as it leads to the interfacial velocity to be equal to one of the phasic velocities.

Following \cite{saurel:1999,saurel:2001}, we consider the mass-weighted average velocity, $\vel=Y_1\vel_1+ Y_2\vel_2$, as interfacial velocity. When drag is considered, this is the velocity to which the phasic velocities relax \cite{saurel:1999}. It is also the choice that is made in the two-phase flow models derived by the theory of thermodynamically compatible systems \cite{romenski:2004,romenski:2007,romenski:2010}.  Moreover, as shown in \cite{gallouet:2004}, under mild assumptions, this is the only all-topology velocity closure which allows the non-conservative products in the Baer-Nunziato model to be uniquely defined in the case of discontinuities. Therefore, there are both good physical and mathematical arguments in favour of this closure. 

However, this new constraint related to the volume fraction cannot be exactly integrated along the phasic trajectories, and as such, no exact expression for the perturbation $\delta\alpha_1$ as a function of $\delta\Traj_1$ and $\delta\Traj_2$ is available. To circumvent this issue, we impose this new constraint through a Lagrange multiplier $\zeta$. This leads to a supplementary contribution, $\zeta\rho\Dt\alpha_1$, to the Lagrangian, which we write as
\begin{equation}
\zeta\rho\Dt\alpha_1 = \zeta m_1\Dt^1\alpha_1 + \zeta m_2\Dt^2\alpha_1.
\end{equation}
Adding this contribution to the physical Lagrangian \eqref{eq:two-phase_Lagrangian}, results in the following effective Lagrangian,
\begin{multline}
\Lag = \half m_1 \vel_1^2 - m_1 e_1\pth{\dfrac{m_1}{\alpha_1},s_1} + \zeta m_1\Dt^1\alpha_1 \\ + \half m_2 \vel_2^2 - m_2 e_2\pth{\dfrac{m_2}{1-\alpha_1},s_2} + \zeta m_2\Dt^2\alpha_1.
\label{eq:lagrangian_with_constraint}
\end{multline}
The previously proposed framework is then applied with $\Lag_1$ and $\Lag_2$ being given by the first and second lines of eq. \eqref{eq:lagrangian_with_constraint}'s right hand-side respectively and where the constraints are given by $\Xiset=\{\alpha_1,\zeta\}$ and $\Bset_k=\{s_k\}$, $k=1,2$.  

Application of Theorem \ref{thrm:SAP} then yields the following equations,
\begin{subequations}
\begin{align}\label{sys:phasic_momentum_vol_frac_zeta}
\partial_t\pth{\alpha_1\rho_1\vel_1} + \nabla\cdot\pth{\alpha_1\rho_1\vel_1\otimes\vel_1 + \alpha_1 p_1\Id} - p_I\nabla\alpha_1 &=+\lift,\\
\partial_t\pth{\alpha_2\rho_2\vel_2} + \nabla\cdot\pth{\alpha_2\rho_2\vel_2\otimes\vel_2 + \alpha_2 p_2\Id} + p_I\nabla\alpha_1 &=-\lift,\\
\partial_t\alpha_1 + \vel\cdot\nabla\alpha_1 &=0,\\
\partial_t \zeta + \vel\cdot\nabla\zeta -\frac{p_1-p_2}{\rho} &=0, \label{eq:SAP_zeta_eq}
\end{align}
\end{subequations}
with the following expressions for the interfacial pressure $p_I$, and the additional momentum exchange force $\lift$,
\begin{subequations}
\begin{align}
p_I &= Y_2 p_1 + Y_1 p_2,\label{eq:interfacial_pressure}\\
\lift &= \rho Y_1 Y_2 \brck{\nabla\zeta\otimes\nabla\alpha_1 - \nabla\alpha_1\otimes\nabla\zeta}\bm{W},
\end{align}
\end{subequations}
where we recall that $\bm{W}=\vel_1-\vel_2$ denotes the relative velocity. These equations are to be supplemented with
\begin{subequations}
\begin{align}\label{sys:mass_entropy_constraints}
\partial_t \pth{\alpha_1\rho_1} + \nabla\cdot\pth{\alpha_1\rho_1 \vel_1} &=0,\\
\partial_t \pth{\alpha_2\rho_2} + \nabla\cdot\pth{\alpha_2\rho_2 \vel_2} &=0,\\
\partial_t s_1 + \vel_1\cdot\nabla s_1 &=0,\\
\partial_t s_2 + \vel_2\cdot\nabla s_2 &=0,
\end{align}
\end{subequations}
to form a complete and closed system. This systems admits a supplementary conservation law for the Hamiltonian, however, as is custom, we replace the phasic entropy equations by the phasic energy equations, as given by Theorem \ref{thrm:Hamiltonian}, and use the total entropy balance as supplementary equation. Moreover, notice that the expression of $\lift$ involves two gradients, $\nabla\alpha_1$ and $\nabla\zeta$, and so the system, as is, is not in the form of a first-order system of balance laws. To recast it in such a form, we introduce $\bm{v}=\nabla\zeta$ as a new independent variable which replaces $\zeta$. The resulting first-order system then writes
\begin{subequations}
\begin{align}
\partial_t\alpha_1 + \vel\cdot\nabla\alpha_1 &=0, \label{syst:alpha_eq}\\
\partial_t \pth{\alpha_1\rho_1} + \nabla\cdot\pth{\alpha_1\rho_1 \vel_1} &=0,\\
\partial_t \pth{\alpha_2\rho_2} + \nabla\cdot\pth{\alpha_2\rho_2 \vel_2} &=0,\\
\partial_t\pth{\alpha_1\rho_1\vel_1} + \nabla\cdot\pth{\alpha_1\rho_1\vel_1\otimes\vel_1 + \alpha_1 p_1\Id} - p_I\nabla\alpha_1 &=+\lift,\\
\partial_t\pth{\alpha_2\rho_2\vel_2} + \nabla\cdot\pth{\alpha_2\rho_2\vel_2\otimes\vel_2 + \alpha_2 p_2\Id} + p_I\nabla\alpha_1 &=-\lift,\\
\partial_t\pth{\alpha_1\rho_1 E_1} + \nabla\cdot\pth{\alpha_1\rho_1 H_1} - \pth{p\vel}_I\cdot\nabla\alpha_1 &= + \vel\cdot\lift,\\
\partial_t\pth{\alpha_2\rho_2 E_2} + \nabla\cdot\pth{\alpha_2\rho_2 H_2} + \pth{p\vel}_I\cdot\nabla\alpha_1 &= - \vel\cdot\lift,\\
\partial_t \bm{v} + \pth{\vel\cdot\nabla}\bm{v} + \pth{\nabla\vel}\bm{v} - \nabla\pth{\frac{p_1-p_2}{\rho}} &= \bm{0},\label{syst:v_eq}
\end{align}\label{syst:energy_form_with_lift}
\end{subequations}
with the new interfacial quantity, the interfacial work $(p\vel)_I$, which is different from the product $p_I\vel_I$, and is defined by
\begin{equation}\label{eq:interfacial_work}
(p\vel)_I = Y_1 p_2 \vel_1 + Y_2 p_1 \vel_2,
\end{equation}
and the right-hand side that is defined by
\begin{equation}\label{eq:lift_expression}
\lift=\rho Y_1 Y_2 \brck{\bm{v}\otimes\nabla\alpha_1 - \nabla\alpha_1\otimes\bm{v}}\bm{W}.
\end{equation}
This system must then be adjoined with the condition $\rot(\bm{v})=\bm{0}$. As this constraint is preserved in time by \eqref{syst:v_eq}, it reduces to a condition on the initial condition for $\bm{v}$. Moreover, one can show that the system is Galilean invariant. Finally, note that for smooth solutions, \eqref{syst:alpha_eq} may be recast into the following conservation law for $\rho \alpha_1$,
\begin{equation}
\partial_t\pth{\rho\alpha_1} + \nabla\cdot\pth{\rho\alpha_1\vel} = 0.
\end{equation}

Before moving on to the discussion about the obtained model's features, we mention that a fully Eulerian derivation is also possible. Stationary Action Principle in the fully Eulerian setting assume independent variations for each variable and include constraints such as mass conservation through Lagrange multipliers. One then needs to add the so-called Lin's constraints to the Lagrangian to enforce the fixed endpoints constraints \cite{bretherton:1970}. Hence, in the Eulerian setting, one considers the following Lagrangian,
\begin{multline}\label{eq:Lagrangian_Lin}
\Lag = \sum_{k=1}^{2}\Bigg\{\tfrac{1}{2}m_k\vel_k^2 - m_k e_k\pth{\frac{m_k}{\alpha_k},s_k} + \zeta m_k\Dt^k\alpha_1 \\
 + \xi_k\brck{\partial_t m_k + \nabla\cdot\pth{m_k\vel_k}} + m_k\theta_k\Dt^k s_k + m_k\sum_{i}\lambda_{k,i}\Dt^k\mu_{k,i}\Bigg\}.
\end{multline}
Its first line corresponds to the effective Lagrangian \eqref{eq:lagrangian_with_constraint}, whereas the second line contains the phasic mass conservation constraint, the phasic entropy transport constraint, and the phasic Lin constraint, imposed respectively by the Lagrange multipliers $\xi_k$, $\theta_k$; and $\lambda_{k,i}$. One then considers independent, smooth, and compactly supported perturbations $\delta\mathcal{X}$, for all variables $\mathcal{X}\in\{m_k,\vel_k,s_k,\alpha_1,\zeta,\xi_k,\theta_k,\lambda_{k,i},\mu_{k,i}\}$. The fully Eulerian Stationary Action Principle yields an evolution equation for each variable, except the phasic velocities which satisfy 
\begin{equation}
\vel_k =  \nabla\xi_k - \theta_k\nabla s_k - \zeta\nabla\alpha_1 - \sum_{i}\lambda_{k,i}\nabla\mu_{k,i}. 
\label{eq:Clebsch_representation_vel_k}
\end{equation}
Combining the various equations then allows to recover the phasic momentum equations and to derive system \eqref{syst:energy_form_with_lift}. The details of this derivation are provided in Appendix \ref{app:derivation_eulerienne}.

The decomposition of the phasic velocities given by eq. \eqref{eq:Clebsch_representation_vel_k} corresponds to a representation in terms of Clebsch potentials, such a decomposition is standard when considering the Stationary Action Principle in an Eulerian framework \cite{bretherton:1970}. The importance of adding Lin's constraints is also highlighted by eq. \eqref{eq:Clebsch_representation_vel_k}. Indeed, assume momentarily that no such constraint was added. Then one obtains \eqref{eq:Clebsch_representation_vel_k} with $\lambda_{k,i}\equiv0$ and $\mu_{k,i}\equiv0$. In particular, in the three-dimensional case, the phasic vorticity then writes
\begin{equation}
\rot(\vel_k)= -\nabla\theta_k \wedge \nabla s_k - \nabla\zeta\wedge\nabla\alpha_1.
\end{equation}
This is a non-physical and restrictive condition that limits the validity of the equations to a particular subset of flows. For instance, in isentropic flows, it restricts the initial conditions to velocity fields which have their vorticity orthogonal to the volume fraction gradient. Without Lin's constraint in the derivation of the Euler equations, a similar non-physical situation occurs where the vorticity is required to be orthogonal to the  gradient of the specific entropy \cite{bretherton:1970}.

\subsection{Analysis of the interfacial quantities}

We now discuss the interfacial quantities present in model \eqref{syst:energy_form_with_lift} derived in the previous paragraph. The supplementary variable $\bm{v}$, and the role of the momentum exchange term $\lift$, will be discussed hereafter.

Recall that model \eqref{syst:energy_form_with_lift} relies on three interfacial quantities, namely an interfacial velocity, chosen as the mass-weighted average velocity $\vel$, an interfacial pressure $p_I$, and an interfacial work term, denoted $(p\vel)_I$. Although the expression of the interfacial velocity is postulated and imposed as a constraint in the variational principle, the obtained closures for the interfacial pressure \eqref{eq:interfacial_pressure}, and interfacial work \eqref{eq:interfacial_work}, result from the variational principle itself. They are neither \textit{a priori} assumptions, nor are they \textit{a posteriori} closures as in the phase-averaged models. Yet, similar closures may be derived through phase averaging, indeed in \cite{ransom:1984}, a one-dimensional phase-averaged model was derived. The phase-average operator was applied to the entropy equations instead of the energy equations. The interfacial closures compatible with the conservation of the total energy were investigated and it was shown that total energy conservation requires the interfacial velocity and interfacial pressure to satisfy
\begin{equation}\label{eq:condition_Ransom}
\vel_1 p_1 - \vel_2 p_2 + (\vel_2-\vel_1)p_I + \vel_I(p_2-p_1)=\bm{0}.
\end{equation}
The authors then proposed $\vel_I=(\vel_1 + \vel_2)/2$, and obtained $p_I=(p_1 + p_2)/2$. However, when inserting $\vel_I=Y_1\vel_1 + Y_2\vel_2$ in \eqref{eq:condition_Ransom}, one obtains $p_I = Y_2 p_1 + Y_1 p_2$, and thus our model is also compatible with the phase-averaging approach proposed in \cite{ransom:1984}. 

The physical interpretation of the interfacial pressure \eqref{eq:interfacial_pressure} is the following. Each phase is subjected to its own pressure gradient, however, when the phases are out of pressure equilibrium, each phase is subjected to a supplementary force located at the interface and proportional to the exterior pressure. For phase $1$, its exterior pressure is $p_2$, whereas for phase $2$, its exterior pressure is $p_1$. Averaging the exterior pressures weighted by the concentrations of the phases upon which they act then yields $p_I = Y_1 p_2 + Y_2 p_1$, i.e. \eqref{eq:interfacial_pressure}, when the concentrations are chosen to be the mass fractions. In the phase averaged model derived in \cite{chen:1996}, a similar argument was made, but the volume fractions were chosen as concentrations instead, leading to the expression $p_I=\alpha_1 p_2 + \alpha_2 p_1$. This argument also provides an interpretation of the interfacial work \eqref{eq:interfacial_work}. If the force induced in phase $1$ by the pressure imbalance is proportional to $p_2$, then its associated work is proportional to $p_2\vel_1$. By symmetry, for phase $2$, we obtain $p_1\vel_2$. Averaging these contributions, in accordance with their respective mass fractions, yields $(p\vel)_I = Y_1 p_2\vel_1 + Y_2 p_1 \vel_2$, which is exactly the expression obtained from the variational principle in \eqref{eq:interfacial_work}.

Interestingly, our expression for the interfacial pressure \eqref{eq:interfacial_pressure} is identical to the one obtained in the models derived through the theory of thermodynamically compatible systems, as well as the interfacial pressure \eqref{eq:p_I_BN} of the all-topology Baer-Nunziato model analyzed in \cite{gallouet:2004} in the isothermal case. However, as discussed in the introduction, the interfacial pressure \eqref{eq:p_I_BN} of the all-topology Baer-Nunziato model strongly depends on the phasic temperatures, while the thermodynamically compatible systems present an artificial heat exchange as well as temperature gradients in their momentum equations. The interfacial work \eqref{eq:interfacial_work}, provided by the Stationary Action Principle, allows to overcome these issues. In comparison, the Baer-Nunziato models do not possess an independent quantity for the interfacial work $(p\vel)_I$, but consider instead the product $p_I\vel_I$, in the phasic energy equations.

Several other contributions also introduced the interfacial work as an independent quantity in addition to the interfacial velocity and interfacial pressure \cite{chen:1996,jin:2006,perrier:2021}. The approaches that have been followed in \cite{chen:1996,jin:2006}, relied on phase-averaging and the interfacial work appeared as one of the unclosed terms, while in \cite{perrier:2021}, a two-pressure two-fluid model has been derived through the averaging of a stochastic model and its closures depends on the assumptions on the stochastic process. In the stochastic derivation, an interfacial work term, different from the product of the interfacial velocity with the interfacial pressure in the general case, has been determined.

The impact of the interfacial work is highlighted in the phasic internal energy equations, as system \eqref{syst:energy_form_with_lift} implies 
\begin{subequations}
\begin{align}\label{eq:internal_energy_eq_all_topology}
\alpha_k\rho_k\Dt^k e_k = -\alpha_k p_k\nabla\cdot\vel_k + p_k\pth{\vel_I-\vel_k}\cdot\nabla\alpha_k .
\end{align}
The first term on the right-hand side of \eqref{eq:internal_energy_eq_all_topology} corresponds to the classical compression effects induced by a non divergence-free velocity field. The second term is an additional effect induced by a difference between the phasic velocity and the interfacial velocity. When $\pth{\vel_I-\vel_k}\cdot\nabla\alpha_k>0$, the $k^\text{th}$ phase is compressed, while otherwise it is dilated.
For the Baer-Nunziato system, the phasic internal energy equations write
\begin{equation}
\alpha_k\rho_k\Dt^k e_k  = -\alpha_k p_k\nabla\cdot\vel_k + p_I\pth{\vel_I-\vel_k}\cdot\nabla\alpha_k.
\end{equation}
As such, for the Baer-Nunziato models, the rate at which phasic internal energies increase, due to compressions caused by a relative velocity, is given by the interfacial pressure instead of the phasic pressures. 
\end{subequations}

Similarly, for the phasic entropy equations, our model \eqref{syst:energy_form_with_lift} results in the transport of the phasic entropies at their own velocity,
\begin{subequations}
\begin{equation}
\alpha_k\rho_k \Dt^k s_k = 0,
\end{equation}
while for the Baer-Nunziato model, one has
\begin{equation}
\alpha_k\rho_k \Dt^k s_k =\frac{1}{T_k}\pth{p_I-p_k}\pth{\vel_I-\vel_k}\cdot\nabla\alpha_k.
\end{equation}
\end{subequations}
In particular, any all-topology Baer-Nunziato model leads to entropy exchanges between the phases and a conservation law for the total entropy can only be achieved through an interfacial pressure that depends on the phasic temperatures \cite{gallouet:2004,muller:2016}. It is precisely the new term related to the interfacial work that allows for a physically relevant model, which we will prove in Section \mbox{\ref{sec:Model_analysis}} to possess all the required fundamental mathematical properties for such PDE systems in the 1D case.

\subsection{Multi-dimensional effects}

We now discuss the forces $\lift$, defined by \eqref{eq:lift_expression}, which appear on the right-hand side of system \eqref{syst:energy_form_with_lift}, and the new variable $\bm{v}$. First, in the one-dimensional setting, this force is zero, and equation \eqref{syst:v_eq}, which governs the evolution of $\bm{v}$, can be discarded. In the general case, this force may be non-zero, but always satisfies $\bm{W}\cdot\lift= 0$, as such it is a force that represents lift. Moreover, because it depends on $\nabla\alpha_1$, this lift force is localized at the interfaces. Its amplitude and direction is partly determined by $\bm{v}=\nabla\zeta$, which has the dimension of a velocity, and thus $\zeta$ plays the role of a velocity potential. We recall that the evolution equation for $\zeta$, given by \eqref{eq:SAP_zeta_eq}, writes
\begin{equation}
\rho\Dt\zeta = p_1 - p_2.
\end{equation}
Consequently, for flows which remain at pressure equilibrium, $\zeta$ will be constant along the fictive trajectories defined by the mixture velocity $\vel$. If at the initial time it is constant, it remains so and the lift force is zero. As a result, expression \eqref{eq:lift_expression} corresponds to lift induced by pressure non-equilibrium. At this stage, the question of how to initialize $\bm{v}=\nabla\zeta$, and its physical meaning, remain under investigation.

The appearance of lift forces in two-phase flow models obtained by variational principles has already been noted, see \cite{romenski:2007} for instance. However, it has been shown in \cite{gavrilyuk:2020}, that for a same physical Lagrangian several systems may be derived, which differ only by the expression of their lift forces. In particular, during the variational process, fluid particle trajectories $\Traj_1$ and $\Traj_2$, are used to define the variations of the velocities $\vel_1$ and $\vel_2$. If instead, one operates a change of variables from $(\vel_1, \vel_2)$ to $(\vel, \vel_1)$ or $(\vel, \vel_2)$, and associates fictive trajectories to these new velocities, then one obtains a new set of equations which differs only by the expression of its lift forces. Consequently, we have used in our derivation only the true fluid particle trajectories associated to $\vel_1$ and $\vel_2$. 

Lift forces in the phasic momentum equations have also been obtained in the Thermodynamically Compatible models \cite{romenski:2004,romenski:2007,romenski:2010}, their expression however does not require any new variable as the forces write
\begin{equation}
\pm\rho Y_1 Y_2 \brck{Y_2\rot\pth{\vel_1} + Y_1\rot\pth{\vel_2}}\wedge\bm{W}. \label{eq:lift_Romenski}
\end{equation}
This expression may be interpreted as the effect of the rotation of one fluid, determined by $\rot(\vel_k)$, on the other phase $k'$. However, the  derivation of these models relies on the change of variables $(\vel_1, \vel_2)$ to $(\vel, \bm{W})$, and the use of a single fictive trajectory associated to the mixture velocity $\vel$, while the relative velocity $\bm{W}$ is expressed through the gradient of the Lagrange multiplier which enforces the phasic mass conservation, as shown in \cite{haegeman_RelVel:2025}. Moreover, in these models, the variational principle imposes $\rot(\bm{W})=\bm{0}$, i.e. $\rot(\vel_1)=\rot(\vel_2)$, and thus the fluids have the same rate of rotation. In our model, this involution constraint is carried by the new variable $\bm{v}$, as $\rot(\bm{v})=\bm{0}$, and the phasic velocity fields are unconstrained. Still, in the absence of clear interpretation for our $\bm{v}$ variable, and in light of the results shown in \cite{gavrilyuk:2020} concerning the impact of fictive trajectories on lift forces, the validity of either \eqref{eq:lift_expression} or \eqref{eq:lift_Romenski} remains unclear.

Yet, lift forces do correspond to a physical and observable reality. In the case of disperse flows, a rotating particle in a carrier fluid is subjected to a lift force, called Magnus force, which at low Reynolds numbers writes \cite{rubinow:1961,crowe:2011}, 
\begin{equation}
\lift_{\text{Magnus}} = \frac{\pi}{8}D^3\rho_1 \pth{\frac{1}{2}\rot(\vel_1)-\bm{\omega}_2}\wedge\bm{W},
\end{equation}
where the index $1$ represents the carrier fluid and $2$ the particle which has diameter $D$, and rotation $\bm{\omega}_2$. However, in order to obtain such effects, one should need to consider the rotational energy, proportional to $\bm{\omega}_k^2$, in the Lagrangian. As this energy has not been accounted for in the two-phase Lagrangian \eqref{eq:two-phase_Lagrangian}, neither \eqref{eq:lift_expression} nor \eqref{eq:lift_Romenski} represent internal degrees corresponding to particle rotations.

A better understanding of the lift forces requires further investigations. Our goal is to provide a two-fluid model with a supplementary conservation law, for which the interfacial pressure does not depend on the temperatures as in eq. \eqref{eq:p_I_BN} for the Baer-Nunziato model, and for which there is no artificial heat exchange as in eq. \eqref{eq:artificial_heat_SHTC} for the class of Thermodynamically Compatible models \cite{romenski:2007,romenski:2010}. As the lift forces are expected to be negligible for near pressure equilibrium flows --- see also the next section where we consider relaxation source terms --- we may consider the model without its lift forces which is also of interest. When doing so, we may discard the supplementary variable $\bm{v}$, for which we have no clear physical interpretation yet, and obtain a model which satisfies our requirements, and has good mathematical properties as will be shown in Section \ref{sec:Model_analysis}. The simplified model then reads
\begin{subequations}
\begin{align}
\partial_t\alpha_1 + \vel\cdot\nabla\alpha_1 &=0,\\
\partial_t \pth{\alpha_1\rho_1} + \nabla\cdot\pth{\alpha_1\rho_1 \vel_1} &=0,\\
\partial_t \pth{\alpha_2\rho_2} + \nabla\cdot\pth{\alpha_2\rho_2 \vel_2} &=0,\\
\partial_t\pth{\alpha_1\rho_1\vel_1} + \nabla\cdot\pth{\alpha_1\rho_1\vel_1\otimes\vel_1 + \alpha_1 p_1\Id} - p_I\nabla\alpha_1 &=\bm{0},\\
\partial_t\pth{\alpha_2\rho_2\vel_2} + \nabla\cdot\pth{\alpha_2\rho_2\vel_2\otimes\vel_2 + \alpha_2 p_2\Id} + p_I\nabla\alpha_1 &=\bm{0},\\
\partial_t\pth{\alpha_1\rho_1 E_1} + \nabla\cdot\pth{\alpha_1\rho_1 H_1} - \pth{p\vel}_I\cdot\nabla\alpha_1 &= 0,\\
\partial_t\pth{\alpha_2\rho_2 E_2} + \nabla\cdot\pth{\alpha_2\rho_2 H_2} + \pth{p\vel}_I\cdot\nabla\alpha_1 &= 0.
\end{align}\label{syst:energy_form_without_lift}
\end{subequations}
We may then conclude that, while our model allows to have both physically coherent closures for the interfacial quantities, the modelling of multi-dimensional effects such as lift requires further efforts. We now move on to the next section which concerns the modelling of dissipative source terms.

\subsection{Modelling source terms}

The Hamiltonian principle provides only reversible systems that are without dissipation. Consequently, any dissipative effects must be added \textit{a posteriori}. Moreover, as the model we have derived allows for full non-equilibrium between the two fluids, similarly to the Baer-Nunziato model, it must be equipped with dissipative source terms in order to provide a meaningful description of two-phase flows. These source terms must be such that in the absence of exterior perturbations, the system is driven towards equilibrium. Furthermore, they must be in agreement with the Second Principle of Thermodynamics and induce dissipation in the form of entropy production. As the prognostic variables and the total entropy of our model \eqref{syst:energy_form_without_lift} are the same as for the Baer-Nunziato model, the same source terms, as determined in \cite{baer:1986,kapila:2001,saurel:2001,muller:2016,perrier:2021}, can be seamlessly considered. Doing so, we obtain
\begin{subequations}
\begin{align}
\partial_t\alpha_1 + \vel\cdot\nabla\alpha_1 &= S_{\text{mec}},\\
\partial_t \pth{\alpha_1\rho_1} + \nabla\cdot\pth{\alpha_1\rho_1 \vel_1} &=0,\\
\partial_t \pth{\alpha_2\rho_2} + \nabla\cdot\pth{\alpha_2\rho_2 \vel_2} &=0,\\
\partial_t\pth{\alpha_1\rho_1\vel_1} + \nabla\cdot\pth{\alpha_1\rho_1\vel_1\otimes\vel_1 + \alpha_1 p_1\Id} - p_I\nabla\alpha_1 -\lift &=+\bm{S}_{\text{kin}},\\
\partial_t\pth{\alpha_2\rho_2\vel_2} + \nabla\cdot\pth{\alpha_2\rho_2\vel_2\otimes\vel_2 + \alpha_2 p_2\Id} + p_I\nabla\alpha_1 +\lift &=-\bm{S}_{\text{kin}},\\
\partial_t\pth{\alpha_1\rho_1 E_1} + \nabla\cdot\pth{\alpha_1\rho_1 H_1} - \pth{p\vel}_I\cdot\nabla\alpha_1 - \vel\cdot\lift &= +S_{\text{th}} -p_I S_{\text{mec}}  + \vel\cdot \bm{S}_{\text{kin}},\\
\partial_t\pth{\alpha_2\rho_2 E_2} + \nabla\cdot\pth{\alpha_2\rho_2 H_2} + \pth{p\vel}_I\cdot\nabla\alpha_1 + \vel\cdot\lift &= -S_{\text{th}} +p_I S_{\text{mec}} - \vel\cdot \bm{S}_{\text{kin}},\\
\partial_t \bm{v} + \pth{\vel\cdot\nabla}\bm{v} + \pth{\nabla\vel}\bm{v} - \nabla\pth{\frac{p_1-p_2}{\rho}} &= \bm{0}, \label{syst:v_eq_thermo}
\end{align}\label{syst:energy_form_with_source_terms}
\end{subequations}
where $S_{\text{mec}}$, $\bm{S}_{\text{kin}}$ and $S_{\text{th}}$ correspond to mechanical, kinematic, and thermal relaxation terms respectively. The total entropy equation then writes
\begin{equation}
\begin{split}
\partial_t\pth{\alpha_1\rho_1 s_1 + \alpha_2 \rho_2 s_2} + \nabla\cdot\pth{\alpha_1\rho_1 s_1 \vel_1 + \alpha_2 \rho_2 s_2 \vel_2} =
 &\pth{\frac{1}{T_1}-\frac{1}{T_2}}S_{\text{th}} \\ & + \pth{\frac{Y_2}{T_1} + \frac{Y_1}{T_2}}(\vel_2-\vel_1)\cdot\bm{S}_{\text{kin}} \\ & + \pth{\frac{Y_1}{T_1} + \frac{Y_2}{T_2}}(p_1-p_2)S_{\text{mec}} ,
\end{split}\label{eq:entropy_balance}
\end{equation}
and therefore, setting
\begin{subequations}
\begin{align}
S_{\text{mec}} & = (p_1-p_2)/\varepsilon_p,\\
\bm{S}_{\text{kin}} & = (\vel_2-\vel_1)/\varepsilon_u ,\\
S_{\text{th}} & = (T_2 - T_1)/\varepsilon_T,
\end{align}
\end{subequations}
then yields a non-negative entropy production for positive relaxation parameters $\varepsilon_p$, $\varepsilon_u$, and $\varepsilon_T$. Several closures for the relaxation parameters have been proposed in the literature, see for instance \cite{petitpas:2009,perrier:2021,schmidmayer:2023,herard:2023}. We refer to \cite{gallouet:2004} (Proposition 2) for a result on the preservation of the volume fraction constraint $\alpha_1\in(0,\,1)$ in the presence of source terms. Furthermore, in flow regimes where some of the associated relaxation time-scales are sufficiently small, reduced order models have been derived assuming that the corresponding flow variables remain at equilibrium \cite{kapila:2001,phd_Labois,lund:2012,hantke:2021}. The same model reductions can be applied to \eqref{syst:energy_form_with_source_terms}. First, near equilibrium, the lift forces play no role. Indeed, at velocity equilibrium, we have $\lift=\bm{0}$, while at pressure equilibrium, eq. \eqref{syst:v_eq_thermo} yields $\bm{v}=\bm{0}$, and thus $\lift=\bm{0}$, if we consider the initial condition $\bm{v}(0,\bm{x})=\bm{0}$.  Moreover, the expression provided by the variational principle for the interfacial work \eqref{eq:interfacial_work}  satisfies
\begin{equation}
(p\vel)_I = p_I \vel_I - Y_1 Y_2(\vel_1 - \vel_2)(p_1-p_2),
\end{equation}
and therefore, when considering the asymptotic expansions near velocity or pressure equilibrium, the interfacial work represents a first order perturbation of a Baer-Nunziato model, and thus the reduced order models obtained at zeroth order are identical. This is a reassuring result for two reasons. First, the closure problem for interfacial quantities is related to non-equilibrium flows only. Second, several of these reduced order models recover some well-known two-phase flow features, such as the so-called Kapila model \cite{kapila:2001}, which is derived assuming velocity and pressure equilibrium, and recovers the non-monotonic mixture sound speed as determined by Wood \cite{wood:1930}.

Unfortunately, unless complete equilibrium is assumed, these reduced order models are known to be either non-hyperbolic \cite{hantke:2021}, or to have non-conservative products which are ill-defined for shocks \cite{kapila:2001}. Therefore, the use of fully out of equilibrium models, such as the one we have derived, combined with relaxation source terms and numerical methods capable of handling their inherent stiffness, such as the one proposed in \cite{herard:2023}, is preferable to obtain convergent numerical simulations.

In the previous section, we discussed the possibility to neglect lift forces. This operation is compatible with the proposed structure for the source terms since the lift forces play no role in the system's thermodynamics. In particular, if we set $\lift=\bm{0}$ and discard eq. \eqref{syst:v_eq_thermo}, the entropy equation \eqref{eq:entropy_balance} remains valid. Furthermore, we may also add source terms to the evolution equation for $\bm{v}$, i.e. eq. \eqref{syst:v_eq_thermo}, without altering the entropy balance \eqref{eq:entropy_balance} as long as they are compatible with the constraint $\rot(\bm{v})=\bm{0}$. In particular, we may consider the relaxation
\begin{equation}
\partial_t \bm{v} + \pth{\vel\cdot\nabla}\bm{v} + \pth{\nabla\vel}\bm{v} - \nabla\pth{\frac{p_1-p_2}{\rho}} = -\frac{1}{\tau_v}\bm{v},
\end{equation}
with a constant time-scale $\tau_v>0$. The model without lift forces can then be obtained by considering the formal limit $\tau_v\to0$.\\

The next section is dedicated to the mathematical analysis of the model, in particular we exhibit the necessary properties for its well-posedness and show that it allows for the proper treatment of weak solutions.

\section{Mathematical analysis of the model: the one-dimensional setting}\label{sec:Model_analysis}

In this section, the mathematical properties of the model are analysed. We refer to textbooks on systems of conservation laws, such as \cite{godlewski:2021}, for definitions of hyperbolicity, Riemann invariants, and the various properties that will be considered. Our analysis will be carried out in the one-dimensional setting and without source terms. In this case, the lift forces are zero and the equation on $\bm{v}$ is discarded, as in \eqref{syst:energy_form_without_lift}. As the system is, to some extent, an extension of the Baer-Nunziato model, many properties verified for the Baer-Nunziato model are also verified for our model. Moreover, the proofs provided are adaptations of existing proofs for the Baer-Nunziato model found in \cite{gallouet:2004,coquel:2013}. Some properties concerning the multi-dimensional case are provided at the end of this section. 

\subsection{Hyperbolicity and characteristic fields}

The development of two-pressure models is mainly motivated by the ill-posedness of the single-pressure models \cite{lhuillier:2013}. Single-pressure models exhibit complex characteristics and are therefore linearly unstable. Consequently, we consider real characteristics to be a minimal requirement for the modelling of two-phase flows. However, several contributions have suggested that instabilities due to non-hyperbolicity may be controlled \cite{ramshaw:1978,vazquez-gonzalez:2016}.  Two-pressure models often have real characteristics \cite{ransom:1984, gallouet:2004}. For the model we have derived, this is specified by the following proposition. 

\begin{proposition}[Hyperbolicity]\label{prop:hyperbolicity} The system admits real eigenvalues which are $u$, $u_k$ and $u_k\pm c_k$, with $k=1,2$, and where $c_k=\partial p_k/\partial{\rho_k}_{|s_k}$, denotes the phasic sound velocity. The corresponding eigenvectors span $\mathbb{R}^7$, and the system is hyperbolic, under the non-resonance condition $u\neq u_k\pm c_k$, for $k=1,2$.
\end{proposition}
\begin{proof}
Through a change of variables, the system is expressed in terms of the variable $\bm{U}=(\alpha_1,u_1,p_1,s_1,u_2,p_2,s_2)^\top$, and writes in the following quasilinear form,
\begin{equation}\label{eq:quasilinear_form}
\partial_t\bm{U} + \mtrx{M}(\bm{U})\partial_x\bm{U}=\bm{0},
\end{equation}
with
\begin{equation}
\mtrx{M}(\bm{U}) = 
\begin{bmatrix}
u & \bm{0}_{1\times3} & \bm{0}_{1\times3} \\
\bm{z}_{\alpha,1}  & \mtrx{M}_1 & \bm{0}_{3\times3}\\
\bm{z}_{\alpha,2}  & \bm{0}_{3\times3} & \mtrx{M}_2
\end{bmatrix},
\end{equation}
and
\begin{equation}
\mtrx{M}_k = \begin{bmatrix}
u_k & 1/\rho_k & 0 \\
\rho_k c_k^2 & u_k & 0\\
0 & 0 & u_k
\end{bmatrix},
\qquad 
\bm{z}_{\alpha,k} = 
\begin{bmatrix}
\tfrac{p_1-p_2}{\rho} \\ (-1)^k\tfrac{\rho_k c_k^2}{\alpha_k}(u-u_k)\\ 0
\end{bmatrix}.
\end{equation}
The block-triangular structure of $\mtrx{M}(\bm{U})$ allows for an immediate computation of the eigenvalues. When the non-resonance condition is satisfied, the matrix $\mtrx{M}(\bm{U})$ is diagonalizable and a matrix of right eigenvectors $\mtrx{R}(\bm{U})$ is given by
\begin{equation}\label{eq:eigenvectors}
\mathbf{R}(\bm{U})= 
\begin{bmatrix}
1 & 0 & 0 & 0 & 0 & 0 & 0 \\
\tfrac{-Y_2 W}{(u_1-u)^2-c_1^2}\pth{\tfrac{p_1-p_2}{\rho} - \tfrac{c_1^2}{\alpha_1}} & 1 & 1 & 0 & 0 & 0 & 0 \\
\tfrac{\rho_1 c_1^2}{(u_1-u)^2-c_1^2}\pth{\tfrac{p_1-p_2}{\rho} - \tfrac{Y_2^2 W^2}{\alpha_1}} & \rho_1 c_1 & -\rho_1 c_1 & 0 & 0 & 0 & 0 \\
0 & 0 & 0 & 1 & 0 & 0 & 0\\
\tfrac{Y_1 W}{(u_2-u)^2-c_2^2}\pth{\tfrac{p_1-p_2}{\rho} + \tfrac{c_2^2}{\alpha_2}} & 0 & 0 & 0 & 1 & 1 & 0\\
\tfrac{\rho_2 c_2^2}{(u_2-u)^2-c_2^2}\pth{\tfrac{p_1-p_2}{\rho} + \tfrac{Y_1^2 W^2}{\alpha_2}}  & 0 & 0 & 0 & \rho_2 c_2 & -\rho_2 c_2 & 0\\
0 & 0 & 0 & 0 & 0 & 0 & 1
\end{bmatrix}.
\end{equation} 
The columns of $\mtrx{R}(\bm{U})$ correspond to eigenvectors $u$, $u_1+c_1$, $u_1-c_1$ ,$u_1$, $u_2+c_2$, $u_2-c_2$ and $u_2$  respectively.
\end{proof}

\begin{remark}
The non-resonance condition also applies to the Baer-Nunziato models \cite{gallouet:2004} and the two-pressure models derived in \cite{ransom:1984}.
\end{remark}
 
\begin{remark}
At resonance, two eigenvalues coincide and an eigenvector may be missing. However, as determined in \cite{ransom:1984}, a complete set of eigenvectors may be available under additional conditions. Assume resonance for the $k^\text{th}$ phase, i.e. $c_k^2=(u-u_k)^2$, if $p_I = p_k - \rho_k c_k^2$, then the system remains hyperbolic. This also holds when both phases are in resonance.
\end{remark}

We now investigate the nature of the different characteristic fields. To do so, we introduce, for each phase $k$, its so-called fundamental derivative $\mathcal{G}_k$ \cite{menikoff:1989}, which is the thermodynamic function defined by
\begin{equation}
\mathcal{G}_k = -\frac{\tau_k}{2}\frac{\partial^3 e_k / \partial{\tau_k^3}_{|s_k}}{\partial^2 e_k / \partial{\tau_k^2}_{|s_k}},
\end{equation}
where $\tau_k=1/\rho_k$, is the phase's specific volume.

The fundamental derivative $\mathcal{G}_k$ plays an important role in the study of the Riemann problem and the condition $\mathcal{G}_k>0$, or even $\mathcal{G}_k>1$, is usually implicitly assumed. For a single fluid, the condition $\mathcal{G}>0$ corresponds to the convexity of the isentropes in the $p-\tau$ plane, with $\tau=1/\rho$, the specific volume. In particular, it ensures that shocks are compressive, while for $\mathcal{G}<0$, the system admits rarefactions shocks instead. The stronger condition, $\mathcal{G}>1$, provides the convexity of the isentropes in the $p-\rho$ plane as well, and ensures that the sound velocity increases along the Hugoniot curves \cite{thompson:1971,menikoff:1989}.
 
We then have the following classification of the characteristic fields.

\begin{proposition}[Characteristic fields]\label{prop:characteristic_fields} When the system is out of resonance, the characteristic fields associated to the convective eigenvalues $u$ and $u_k$, for $k=1,2$, are linearly degenerate. Moreover, for $k=1,2$, if the equation of state of the $k^\text{th}$ phase is such that $\mathcal{G}_k\neq 0$, then the characteristic fields associated to the acoustic eigenvalues $u_k\pm c_k$, are genuinely non-linear.
\end{proposition}
\begin{proof}
The linearly degenerate character of the eigenvalues $u_1$ and $u_2$ is immediate owing to the expression of their eigenvectors given by \eqref{eq:eigenvectors}. For $k=1,2$, let $\lambda_k(\bm{U})=u_k\pm c_k$ and $\bm{r}_k(\bm{U})$ be the associated eigenvector as given by \eqref{eq:eigenvectors}. We then have that
\begin{equation}
\nabla_{\bm{U}}\lambda_k\cdot\bm{r}_k = 1 + \rho_k c_k \frac{\partial c_k}{\partial p_k}_{|s_k} = 1 + \frac{1}{2}\rho_k\frac{\tfrac{\partial^2 p_k}{\partial \rho_k^2}_{|s_k}}{\tfrac{\partial p_k }{ \partial \rho_k}_{|s_k}} = \mathcal{G}_k,
\end{equation}
which shows the genuinely non-linear character of the acoustic waves. As the classification of the characteristic fields does not depend on the choice of variables, to show the linear degeneracy of the interfacial characteristic field, we use the set of variables $\bm{V}=(\alpha_1, u_1, p_1, m_1, u_2, p_2, m_2)$ with $m_k=\alpha_k \rho_k$. With these variables, the eigenvector $\bm{r}$ associated to the eigenvalue $u$ is given by the first column of \eqref{eq:eigenvectors} but where its two zero entries are respectively replaced by $\frac{\rho Y_1}{(u_1-u)^2-c_1^2}((p_1-p_2)/\rho - c_1^2/\alpha_1)$ and $\frac{\rho Y_2}{(u_2-u)^2-c_2^2}((p_1-p_2)/\rho + c_2^2/\alpha_2)$. Moreover, since 
$\nabla_{\bm{V}}u=(0, Y_2 W/\rho, Y_1, 0, -Y_1 W/\rho, Y_2, 0)^\top$, a direct computation yields $\nabla_{\bm{V}}u\cdot\bm{r}=0$.
\end{proof}

\subsection{Riemann invariants and jump conditions}

Solutions to the Riemann problem consist of constant states connected through either contact discontinuities associated to linearly degenerate waves, or shock or rarefaction waves associated to genuinely non-linear waves.
States that can be connected to each other through contact discontinuities or rarefactions waves can be parametrized by the Riemann invariants of the corresponding wave. Through such waves, the Riemann invariants remain constant.

\begin{proposition}[Riemann invariants]
Outside of resonance, each characteristic field admits a set of 6 Riemann invariants with linearly independent gradients. For each characteristic field, such a set of Riemann invariants is given by
\begin{subequations}
\begin{align}
u:\quad&\acco{u, \; \rho Y_1 Y_2 W, \; \alpha_1 p_1 + \alpha_2 p_2 + \rho Y_1 Y_2 W^2,\; h_1 - h_2 + \tfrac{1}{2}(Y_2 - Y_1)W^2, \; s_1,\; s_2},\label{eq:Riemann_invariants_interface}\\
u_1:\quad&\acco{\alpha_1,\; u_1,\; p_1,\; u_2,\; p_2,\; \rho_2} ,\\
u_1+c_1:\quad&\acco{\alpha_1,\;s_1,\; u_1-f_1(s_1,p_1) ,\;u_2,\; p_2,\; \rho_2} ,\\
u_1-c_1:\quad&\acco{\alpha_1,\;s_1,\: u_1+f_1(s_1,p_1) ,\;u_2,\; p_2,\; \rho_2} ,\\
u_2:\quad&\acco{\alpha_1,\; u_2,\; p_2,\; u_1,\; p_1,\; \rho_1} ,\\
u_2+c_2:\quad&\acco{\alpha_1,\; s_2,\; u_2-f_2(s_2,p_2),\; u_1,\; p_1,\; \rho_1} ,\\
u_2-c_2:\quad&\acco{\alpha_1,\; s_2,\; u_2+f_2(s_2,p_2),\; u_1,\; p_1,\; \rho_1} ,
\end{align}
where $f_k(s_k,p_k) = \int^{p_k} \brck{\rho_k (s_k,p )c_k(s_k,p )}^{-1}\mathrm{d}p$, for $k=1,2$.
\end{subequations}
\end{proposition}
\begin{proof}
The Riemann invariants are obtained by direct computation using the expression of the eigenvectors \eqref{eq:eigenvectors}. For some of the Riemann invariants associated to the interface wave \eqref{eq:Riemann_invariants_interface}, it is more practical to use the variables $\bm{V}$, and the associated eigenvector, introduced in the proof of proposition \ref{prop:characteristic_fields}, in combination with the following thermodynamic identity
\begin{equation}
c_k^2\dfrac{\partial h_k}{\partial p_k}_{|\rho_k} + \dfrac{\partial h_k}{\partial \rho_k}_{|p_k} = \frac{c_k^2}{\rho_k}.
\end{equation}
\end{proof}

\begin{remark}
The interfacial pressure $p_I$, given by eq. \eqref{eq:interfacial_pressure}, is not continuous through the interfacial wave. Instead, it is a total pressure $\alpha_1 p_1 + \alpha_2 p_2 + \rho Y_1 Y_2 W^2$ which is continuous. This expression was chosen as interfacial pressure for a Baer-Nunziato model in \cite{saurel:2001}, but leads to a non-conservative equation for the total entropy equation and thus the model lacks an admissibility criterion for shocks. 
\end{remark}

The Riemann invariants are identical to the ones of the Baer-Nunziato system, see \cite{gallouet:2004}, except for those associated to the interfacial wave $u$.

In addition, similarly to the Baer-Nunziato models, the two-phases are coupled only through the non-conservative products which involve $\partial_x\alpha_1$. Owing to its linear degeneracy, the non-conservative products through the interfacial wave $u$ are entirely defined by the Riemann invariants \eqref{eq:Riemann_invariants_interface}. Moreover, the volume fraction remains constant on both sides of the interfacial wave, and thus the non-conservative products are inactive through shocks and the phases are decoupled \cite{gallouet:2004}. In particular, if we denote by $\jump{\phi}$ the jump of a variable $\phi$ through a discontinuity, we obtain the following jump conditions for the genuinely non-linear waves.

\begin{proposition}[Jump conditions]
For the genuinely non-linear waves, two states connected through a shock, propagating at speed $\sigma$, satisfy the following jump conditions,
\begin{subequations}
\begin{align}
\jump{\alpha_1} &= 0,\\
\jump{\alpha_k \rho_k\pth{u_k-\sigma}} &=0,\\
\jump{\alpha_k \rho_k u_k\pth{u_k-\sigma}} + \jump{\alpha_k p_k} &=0,\\
\jump{\alpha_k \rho_k E_k\pth{u_k-\sigma}} + \jump{\alpha_k p_k u_k} &=0,
\end{align}
for $k=1,2$.
\end{subequations}
\end{proposition}
\begin{proof}
Direct application of the Rankine-Hugoniot jump conditions when $\partial_x\alpha_1 =0$.
\end{proof}

In Section \ref{sec:Model_derivation}, we have imposed the transport of specific entropies for each phase as a constraint in the variational derivation of the model. Consequently, the model admits a supplementary conservation law for the physical entropy which allows to select physically admissible shocks.
A shock is said to be admissible if it satisfies the following entropy inequality, 
\begin{equation}\label{eq:entropy_inequality}
\partial_t\pth{\alpha_1\rho_1 s_1 + \alpha_2 \rho_2 s_2} + \partial_x \pth{\alpha_1\rho_1 s_1 u_1 + \alpha_2 \rho_2 s_2 u_2} \geq 0,
\end{equation}
and thus the mathematical entropy is given by the opposite of the physical entropy. Note that when written in terms of mixture velocity $\vel$ and relative velocity $\bm{W}$, eq. \eqref{eq:entropy_inequality} corresponds to \eqref{eq:entropy_eq_u_W}.

\begin{remark}
The mathematical entropy is convex, but not strictly convex as shown in \cite{coquel:2013}.
\end{remark}

For the Baer-Nunziato models, which have $(p\vel)_I=p_I \vel_I$, the availability of an entropy conservation law is not guaranteed and depends on the closures for the interfacial quantities $\vel_I$ and $p_I$ \cite{gallouet:2004}. In particular, with the Baer-Nunziato model, when considering the interfacial velocity to be the mass-weighted average velocity as we do, $\vel_I=\vel$, then a conservation law for the system's entropy is obtained only with the closure \eqref{eq:p_I_BN} for the interfacial pressure $p_I$. As this closure depends explicitly on the phasic temperatures, it results that the solution to the Riemann problem depends on quantities such as the phasic heat capacities $c_{v,k} = \partial e_k / \partial {T_k}_{|\rho_k}$. Such a dependence is unexpected in classical mechanics, for instance, the solution of the Riemann problem for the Euler equations does not depend on such thermal quantities. In our model, the closure \eqref{eq:interfacial_work} provided by the variational principle for the interfacial work $(p\vel)_I$ allows to overcome this issue as it provides the mathematical properties for the proper treatment of the Riemann problem, without inducing non-physical behaviour.

\subsection{Symmetrization}

The previous results concern weak solutions, and in particular, we have shown that shocks are uniquely defined despite the non-conservative nature of the system. Here, we tackle the question of its well-posedness for smooth initial data, in the sense of Sobolev regularity. Friedrichs introduced an important class of symmetric hyperbolic systems \cite{friedrichs:1954}, for which local-in-time well-posedness results from Kato's theorem \cite{kato:1975}. These results then extend to symmetrizable systems. 

In \cite{coquel:2013}, the Baer-Nunziato models have been shown to be symmetrizable when out of resonance, independently of the closures for the interfacial quantities. Although the derived model is not a Baer-Nunziato model, it is, in some sense, an extension of this class of models and thus we have the following proposition.

\begin{proposition}[Symmetrizability]\label{prop:symmetrization}
 The system is symmetrizable when the non-resonance condition is satisfied.
\end{proposition}
\begin{proof}
We follow and adapt the proof of the symmetrizability of the Baer-Nunziato system proposed in \cite{coquel:2013}. We express the system in the same quasilinear form \eqref{eq:quasilinear_form} as in the proof of proposition \ref{prop:hyperbolicity} and exhibit a symmetric positive definite matrix $\mtrx{P}(\bm{U})$ such that $\mtrx{P}(\bm{U})\mtrx{M}(\bm{U})$ is symmetric. A symmetric form of the system is then given by
\begin{equation}
\mtrx{P}(\bm{U})\partial_t\bm{U} + \mtrx{P}(\bm{U})\mtrx{M}(\bm{U})\partial_x\bm{U}=\bm{0}.
\end{equation}
Setting $\mtrx{P}_k = \mathrm{diag}(\rho_k c_k^2,1,1)$, we consider 
\begin{equation}
\mtrx{P} = \begin{bmatrix}
 \theta_{\alpha} & \bm{y}_{1,\alpha}^\top & \bm{y}_{2,\alpha}^\top \\
 \bm{y}_{1,\alpha} & \mtrx{P}_1 & \bm{0}_{3\times3}\\
 \bm{y}_{2,\alpha} & \bm{0}_{3\times3} & \mtrx{P}_2
 \end{bmatrix},
\end{equation}
where $\theta_\alpha\in\mathbb{R}$ and $\bm{y}_{1,\alpha},\,\bm{y}_{2,\alpha}\in\mathbb{R}^3$ are to be determined. The requirement $\mtrx{P}(\bm{U})\mtrx{M}(\bm{U})$ symmetric is equivalent to
\begin{equation}
\pth{\mtrx{M}_k^\top - u\mtrx{I}_3}\bm{y}_{k,\alpha} = \mtrx{P}_k\bm{z}_{k,\alpha}, \qquad k=1,2.
\end{equation}
This is a $3\times3$ linear system whose last line yields a trivial equation and allows to set $\bm{y}_{k,\alpha}^{(3)}=0$, the remaining $2\times2$ system writes
\begin{equation}
\begin{bmatrix}
u_k- u & \rho_k c_k^2 \\
1/\rho_k & u_k - u
\end{bmatrix}
\begin{bmatrix}
\bm{y}_{k,\alpha}^{(1)}\\
\bm{y}_{k,\alpha}^{(2)}
\end{bmatrix}
= \begin{bmatrix}
1 & 0 & 0\\
0 & 1 & 0 
\end{bmatrix}
\mtrx{M}_k \bm{z}_{k,\alpha}, \qquad k=1,2.
\end{equation}
When the non-resonance condition is satisfied, this linear system is invertible and yields a closed expression for $\bm{y}_{k,\alpha}$, for $k=1,2$, such that $\mtrx{P}(\bm{U})\mtrx{M}(\bm{U})$ is symmetric. It remains to determine $\theta_\alpha$ such that $\mtrx{P}(\bm{U})$ is positive definite. For any $\bm{X}=(a,\bm{b}_1^\top,\bm{b}_2^\top)\in\mathbb{R}^7$, we have that
\begin{equation}
\bm{X}^\top \mtrx{P}(\bm{U}) \bm{X}  
= a^2 \theta_\alpha + 2a\pth{\bm{y}_{1,\alpha}^\top\bm{b}_1 + \bm{y}_{2,\alpha}^\top\bm{b}_2} + \pth{\bm{b}_1^\top\mtrx{P}_1\bm{b}_1 + \bm{b}_2^\top\mtrx{P}_2\bm{b}_2},\label{eq:spd_polynome}
\end{equation}
This is a second-order polynomial with respect to $a$ and its discriminant $\Delta$ writes
\begin{equation}
\Delta = 4\brck{\pth{\bm{y}_{1,\alpha}^\top\bm{b}_1 + \bm{y}_{2,\alpha}^\top\bm{b}_2}^2 - \theta_\alpha\pth{\bm{b}_1^\top\mtrx{P}_1\bm{b}_1 + \bm{b}_2^\top\mtrx{P}_2\bm{b}_2}}.
\end{equation}
As $\mtrx{P}_k$ is symmetric positive definite, its square root, $\mtrx{P}_k^{1/2}$, is also symmetric positive definite and we may set $\overline{\bm{b}}_k = \mtrx{P}_k^{1/2}\bm{b}_k$, for $k=1,2$, such that
\begin{equation}
\frac{\Delta}{4} = \pth{\bm{y}_{1,\alpha}^\top\mtrx{P}_1^{-1/2}\overline{\bm{b}}_1 + \bm{y}_{2,\alpha}^\top\mtrx{P}_2^{-1/2}\overline{\bm{b}}_2}^2 - \theta_\alpha \pth{ |\overline{\bm{b}}_1|^2 + |\overline{\bm{b}}_2|^2}.
\end{equation}
Applying the Cauchy-Schwarz inequality in $\mathbb{R}^6$ to $\pth{\overline{\bm{b}}_1^\top,\overline{\bm{b}}_2^\top}$ and $\pth{(\mtrx{P}_1^{-1/2}\bm{y}_{1,\alpha})^\top,(\mtrx{P}_2^{-1/2}\bm{y}_{2,\alpha}^\top)}$ then yields
\begin{equation}
\frac{\Delta}{4} \leq  \pth{\left|\mtrx{P}_1^{-1/2}\bm{y}_{1,\alpha}\right|^2 + \left|\mtrx{P}_2^{-1/2}\bm{y}_{2,\alpha}\right|^2 - \theta_\alpha} \pth{ |\overline{\bm{b}}_1|^2 + |\overline{\bm{b}}_2|^2}.
\label{eq:majoration_discriminant}
\end{equation}
Choosing $\theta_\alpha >\left|\mtrx{P}_1^{-1/2}\bm{y}_{1,\alpha}\right|^2 + \left|\mtrx{P}_2^{-1/2}\bm{y}_{2,\alpha}\right|^2$ results in a non-positive discriminant so that $\bm{X}^\top\mtrx{P}(\bm{U})\bm{X}\geq0$. Moreover, if $\bm{X}^\top\mtrx{P}(\bm{U})\bm{X}=0$, then necessarily $\Delta=0$ and so by \eqref{eq:majoration_discriminant}, $\overline{\bm{b}}_k=0$. It is then immediate that $\bm{b}_k=0$, and by \eqref{eq:spd_polynome} it follows that $a=0$, so that $\bm{X}=0$. As a result, we may choose $\theta_\alpha$ such that $\mtrx{P}(\bm{U})$ is symmetric positive definite and symmetrizes the quasilinear form \eqref{eq:quasilinear_form} of the system.
\end{proof}

\begin{remark}
For conservative systems, the existence of a supplementary conservation law for a strictly convex entropy implies symmetrizability of the system \cite{godunov:2025,friedrichs:1971}. However, our model does not admit a conservative form, nor is its entropy strictly convex. In \cite{forestier:2011}, the symmetrizability of non-conservative systems having a supplementary conservation law has been investigated.  It was shown that strict convexity of the mathematical entropy is required only for a subset of the system's variables, but additional conditions need to be satisfied to deal with the non-conservative terms. In the case of our model, this additional condition corresponds to the non-resonance condition. Here, the supplementary conservation law is provided by variational principle, however, for more general conditions concerning the existence of an additional conservation law for non-conservative systems, we refer to \cite{cordesse:2020}.
\end{remark}

As a result of its symmetrizability, for initial conditions far from resonance, the Cauchy problem admits a local-in-time smooth solution \cite{kato:1975}. The smooth solution exists until it either becomes resonant or has derivatives which become unbounded in $L^\infty$.

\subsection{Some comments on the multi-dimensional case}

Owing to the model's rotational invariance, its hyperbolicity in the multi-dimensional case may be studied only in the $x$-direction. If the lift forces are neglected and one considers the simplified system \eqref{syst:energy_form_without_lift}, then the model remains hyperbolic. It has the same eigenvalues as in the one-dimensional case but the multiplicity of the convective waves $\vel_{1,x}$ and $\vel_{2,x}$ is increased by the number of tangential components. Moreover, the nature of the fields remains identical and thus the previous discussion concerning shocks and non-conservative products remains valid.

If one considers the system including lift forces, i.e. system \eqref{syst:energy_form_with_lift}, then the model is only weakly hyperbolic. Again, the model has the same eigenvalues as in the one-dimensional case, and the multiplicity of the convective waves $\vel_{1,x}$ and $\vel_{2,x}$ is also increased by the number of tangential components. However, owing to the additional equation \eqref{syst:v_eq} for the supplementary variable $\bm{v}$, the multiplicity of the interfacial eigenvalue $u_x$ is also increased but one eigenvector is missing and we do not have a full basis of eigenvectors. 

This is similar to the Thermodynamically Compatible two-phase models \cite{romenski:2007,romenski:2010}, which are also weakly hyperbolic in the multi-dimensional case. In fact, weak hyperbolicity is recurrent for systems which have an involution constraint such as $\rot(\bm{v})=\bm{0}$, see for instance \cite{romenski:2010,dhaouadi:2024}. Several techniques to restore hyperbolicity have been developed, such as curl-cleaning \cite{dumbser:2020}, and Godunov-Powell symmetrization \cite{godunov:2025b,powell:1999} but have not been investigated in the present work.

\section{Conclusion}\label{sec:Conclusions}
In this paper, the generic Stationary Action Principle framework presented in \cite{haegeman_GenSAP:2025} has been extended to the case of two families of trajectories. The new framework allows for the derivation of the equations of motion in two-velocity models for two-phase flows and guarantees the conservation of total momentum and total energy. 

The newly developed framework is then used for the derivation of a novel \textit{all-topology} two-fluid model. The model is reminiscent of the class of Baer-Nunziato models that it extends through the introduction of a new interfacial quantity, namely the interfacial work. However, contrary to the Baer-Nunziato models, for which many different closures have been considered, our model is fully closed as the sole assumption on the interfacial velocity, in the variational principle, suffices to close the remaining interfacial quantities. The interfacial velocity is given by the mixture velocity, i.e. the mass-weighted average velocity, allowing it to be well-defined and have a meaningful expression regardless of flow conditions. Moreover, the resulting expressions for the interfacial pressure and interfacial work are physically sound and allow to clearly distinguish the various mechanical effects from the thermodynamic ones. 

In the one-dimensional setting, the model is shown to be hyperbolic under a classical non-resonance condition and is symmetrizable, allowing for local-in-time well-posedness for smooth solutions. Moreover, despite its non-conservative nature, its non-conservative products are uniquely defined through discontinuities and an entropy conservation law provides a selection criteria for physically admissible shocks. 

In the multi-dimensional setting, lift forces are present in the phasic momentum equations and are driven by a new potential velocity which has its own evolution equation and the resulting model is weakly hyperbolic. However, in the absence of rotational energies in the Lagrangian, a clear interpretation of the obtained lift forces is not yet available and requires further investigation. Consequently, as they vanish at equilibrium, we may choose to neglect these forces for flows near equilibrium at the present time, and in doing so, we obtain a simplified model which inherits the mathematical properties proved in the one-dimensional case.

It is, to the best of our knowledge, the first \textit{all-topology} multi-fluid model that has both a complete physical consistency and the required mathematical structure to properly describe shock-type solutions. We have thus been able to reconcile approaches based primarily on mathematics and approaches based primarily on physics to obtain an all-topology model to describe the complexity of compressible two-phase flows that potentially involve shocks.

 The contribution of this paper is the development of a model, but this model lends itself well to simulations even though this is not within the scope of the paper. Readers can find some simple numerical simulations at the address indicated\footnote{\url{https://hpc-maths.github.io/2025_09_two_fluid_all_topology/} }, along with the codes to reproduce them. Future works will concern the development of accurate approximate Riemann solvers for the model and further investigations in the modelling of lift forces. \vskip6pt

\enlargethispage{20pt}

\ack{The authors would like to thank Jean-Marc Hérard for fruitful discussions on related topics. We also thank the two anonymous reviewers for their comments. This work has been supported by the Agence Innovation D{\'e}fense (AID), the CIEDS project OPEN NUM DEF (PI. L. Gouarin, M. Massot, T. Pichard), a PhD Grant from ONERA (W. Haegeman), as well as the HPC@Maths Initiative (PI. L. Gouarin and M. Massot) of the Fondation École polytechnique, which we would like to acknowledge.}

\appendix

\section{Derivation of the generic system through Stationary Action principle}\label{app:SAP_proof}

\subsection{Proof of Theorem \ref{thrm:SAP}}
We provide here a proof of the generic equations \eqref{eq:generic_system_Euler_Lagrange} obtained by the Stationary Action Principle \eqref{eq:stationary_action}, for the generic Lagrangian \eqref{eq:generic_Lagrangian}. The variation of the action $\Act$ is given by
\begin{multline}\label{A:variation_of_action}
\delta\Act = \sum_{k=1}^{2}\txint \acco{\dfrac{\partial\Lag_k}{\partial m_k}\delta m_k + \dfrac{\partial\Lag_k}{\partial\vel_k}\delta\vel_k + \sum_{i=1}^{n_k}\dfrac{\partial\Lag_k}{\partial b_{k,i}}\delta b_{k,i} + \sum_{i=1}^{n_{k'}} \dfrac{\partial\Lag_{k}}{\partial b_{k',i}}\delta b_{k',i}} \mathrm{d}\bm{x}\mathrm{d}t\\
 + \sum_{k=1}^{2}\sum_{j=1}^{m}\txint \acco{\dfrac{\partial\Lag_k}{\partial\xi_j}\delta\xi_j + \dfrac{\partial\Lag_k}{\partial\pth{\nabla\xi_j}}\delta\pth{\nabla\xi_j} + \dfrac{\partial\Lag_k}{\partial\pth{\Dt^k\xi_j}}\delta\pth{\Dt^k\xi_j}}\mathrm{d}\bm{x}\mathrm{d}t.
\end{multline}
Here and in the sequel, $k'=3-k$ and is the opposite phase of phase $k$. For the sake of convenience on the notation, the differential element $\mathrm{d}\bm{x}\mathrm{d}t$ will be dropped in the remaining of the proof. The perturbations appearing in \eqref{A:variation_of_action} do not form free family as they are related to the independent perturbations $\delta\Traj_1$, $\delta\Traj_2$ and $\delta\xi_j$ through eqs. \eqref{eq:SAP_variations}. Using these expressions, the various terms in the previous equation write,
\begin{subequations}
\begin{align}
&\txint\dfrac{\partial\Lag_k}{\partial m_k}\delta m_k =  -\txint\dfrac{\partial\Lag_k}{\partial m_k}\nabla\cdot\pth{m_k\delta\Traj_k} = \txint\delta\Traj_k^\top m_k\nabla\pth{\dfrac{\partial\Lag_k}{\partial m_k}} ,\\
&\txint\dfrac{\partial\Lag_k}{\partial b_{\flat,i}} \delta b_{\flat,i} = - \txint\delta\Traj_\flat^\top\dfrac{\partial\Lag_k}{\partial b_{\flat,i}} \nabla b_{\flat,i} , \qquad \flat=k,k',
\end{align}
and
\begin{equation}
\begin{split}
\txint\dfrac{\partial\Lag_k}{\partial\vel_k}\delta\vel_k &= \txint\dfrac{\partial\Lag_k}{\partial\vel_k}\brck{\Dt^k\pth{\delta\Traj_k}-\pth{\partial_{\bm{x}}\vel_k}\delta\Traj_k}\\
&= -\txint\delta\Traj_k^\top\brck{\partial_t\bm{K}_k + \nabla\cdot\pth{\bm{K}_k\otimes\vel_k} + \pth{\nabla\vel_k}\bm{K}_k}.
\end{split}
\end{equation}
\end{subequations}
Recall that the perturbations are assumed to vanish at the boundary and therefore no boundary integral contribution arises. The independent variables contribute through the following terms
\begin{equation}
\txint\dfrac{\partial\Lag_k}{\partial\pth{\nabla\xi_j}}\delta\pth{\nabla\xi_j} = \txint\dfrac{\partial\Lag_k}{\partial\pth{\nabla\xi_j}}\nabla\pth{\delta\xi_j} = -\txint\delta\xi_j\nabla\cdot\pth{\bm{D}_{k,j}},
\end{equation}
and
\begin{equation}
\begin{split}
\txint\dfrac{\partial\Lag_k}{\partial\pth{\Dt^k\xi_j}}\delta\pth{\Dt^k\xi_j} &= \txint\mathscr{M}_{k,j}\Dt^k\pth{\delta\xi_j} + \mathscr{M}_{k,j}\nabla\xi_j^\top\brck{\Dt^k\pth{\delta\Traj_k} - \pth{\partial_{\bm{x}}\vel_k}\delta\Traj_k}\\
&= -\txint\delta\Traj_k^\top\bigg[\partial_t\pth{\mathscr{M}_{k,j}\nabla\xi_j} + \nabla\cdot\pth{\mathscr{M}_{k,j}\nabla\xi_j\otimes\vel_k} \\
& \hspace{40pt}  + \pth{\nabla\vel_k}\mathscr{M}_{k,j}\nabla\xi_j\bigg] +\delta\xi_j\brck{\partial_t\mathscr{M}_{k,j}+\nabla\cdot\pth{\mathscr{M}_{k,j}\vel_k}}.
\end{split}
\end{equation}
Combining these expressions, we get
\begin{multline}
\delta\Act = -\sum_{k=1}^{2}\txint\delta\Traj_k^\top\bigg\{ \partial_t\bm{K}_k + \nabla\cdot\pth{\bm{K}_k\otimes\vel_k} + \pth{\nabla\vel_k}\bm{K}_k 
+ \sum_{i=1}^{n_k}\brck{\dfrac{\partial\Lag_k}{\partial b_{k,i}} +\dfrac{\partial \Lag_{k'}}{\partial b_{k,i}}}\nabla b_{k,i} \\
- m_k\nabla\pth{\dfrac{\partial\Lag_k}{\partial m_k}}  + \sum_{j=1}^{m}\brck{ \partial_t\pth{\mathscr{M}_{k,j}\nabla\xi_j} + \nabla\cdot\pth{\mathscr{M}_{k,j}\nabla\xi_j\otimes\vel_k}  + \pth{\nabla\vel_k}\mathscr{M}_{k,j}\nabla\xi_j}\bigg\} \\
- \sum_{j=1}^{m}\txint\delta\xi_j\bigg\{\partial_t\pth{\mathscr{M}_{1,j}+\mathscr{M}_{2,j}}+\nabla\cdot\pth{\mathscr{M}_{1,j}\vel_1 + \mathscr{M}_{2,j}\vel_2 + \bm{D}_{1,j} + \bm{D}_{2,j}} \\ -\dfrac{\partial\Lag_1}{\partial\xi_j} -\dfrac{\partial\Lag_2}{\partial\xi_j} \bigg\}
\end{multline}
The critical points of the Action then satisfy $\delta\Act=0$, for any compactly supported perturbations $\delta\Traj_1$, $\delta\Traj_2$ and $\delta\xi_j$. In particular, the perturbations $\delta\xi_j$ yield the generic equations \eqref{eq:generic_system_xi_eq} for $j=1,\ldots,m$. The perturbations $\delta\Traj_k$, $k=1,2$, yield
\begin{multline}\label{A:raw_K_eq}
\partial_t\bm{K}_k + \nabla\cdot\pth{\bm{K}_k\otimes\vel_k} + \pth{\nabla\vel_k}\bm{K}_k - m_k\nabla\pth{\dfrac{\partial\Lag_k}{\partial m_k}}  + \sum_{i=1}^{n_k}\brck{\dfrac{\partial\Lag_k}{\partial b_{k,i}} +\dfrac{\partial \Lag_{k'}}{\partial b_{k,i}}}\nabla b_{k,i} \\
 + \sum_{j=1}^{m} \brck{\partial_t\pth{\mathscr{M}_{k,j}\nabla\xi_j} + \nabla\cdot\pth{\mathscr{M}_{k,j}\nabla\xi_j\otimes\vel_k}  + \pth{\nabla\vel_k}\mathscr{M}_{k,j}\nabla\xi_j} = 0.
\end{multline}
The sum indexed by $j$ on the second line of the previous relation simplifies into
\begin{equation}
\sum_{j=1}^{m}\acco{\brck{\partial_t\mathscr{M}_{k,j} + \nabla\cdot\pth{\mathscr{M}_{k,j}\vel_k}}\nabla\xi_j + \mathscr{M}_{k,j}\nabla\pth{\Dt^k\xi_j}}.
\end{equation}
Moreover, we have that
\begin{multline}
\nabla\Lag_k = \dfrac{\partial\Lag_k}{\partial m_k}\nabla m_k + \pth{\nabla\vel_k}\bm{K}_k + \sum_{i=1}^{n_k}\dfrac{\partial\Lag_k}{\partial b_{k,i}}\nabla b_{k,i} + \sum_{i=1}^{n_{k'}}\dfrac{\partial\Lag_k}{\partial b_{k',i}}\nabla b_{k',i}  \\
+ \sum_{j=1}^{m}\acco{\dfrac{\partial\Lag_k}{\partial\xi_j}\nabla\xi_j + \pth{\nabla\nabla\xi_j}\bm{D}_{k,j} + \mathscr{M}_{k,j}\nabla\pth{\Dt^k\xi_j}}.
\end{multline}
Plugging these expressions into \eqref{A:raw_K_eq}, we obtain
\begin{multline}
\partial_t\bm{K}_k + \nabla\pth{\bm{K}_k\otimes\vel_k} + \nabla\pth{\Lag_k-m_k\dfrac{\partial\Lag_k}{\partial m_k}}   -\sum_{i=1}^{n_{k'}}\dfrac{\partial\Lag_k}{\partial b_{k',i}}\nabla b_{k',i} + \sum_{i=1}^{n_k}\dfrac{\partial\Lag_{k'}}{\partial b_{k,i}}\nabla b_{k,i} \\
+\sum_{j=1}^{m}\acco{\brck{\partial_t \mathscr{M}_{k,j} + \nabla\cdot\pth{\mathscr{M}_{k,j}\vel_k}  - \dfrac{\partial\Lag_k}{\partial\xi_j}}\nabla\xi_j - \pth{\nabla\nabla\xi_j}\bm{D}_{k,j}} = 0.
\end{multline}
We then use the fact that $\pth{\nabla\nabla\xi_j}\bm{D}_{k,j} = \nabla\cdot\pth{\nabla\xi_j\otimes\bm{D}_{k,j}} - \pth{\nabla\cdot\bm{D}_{k,j}}\nabla\xi_j$ and rearrange the terms to obtain the generic momentum equations \eqref{eq:generic_system_K1_eq} and \eqref{eq:generic_system_K2_eq}.

\subsection{Proof of Theorem \ref{thrm:Hamiltonian}}
We provide a proof of the conservation law \eqref{eq:generic_hamiltonian_eq} for the Hamiltonian verified by the generic system. We derive the equations for the partial Hamiltonians \eqref{eq:generic_partial_hamiltonian_eq}. Equation \eqref{eq:generic_hamiltonian_eq} is then obtained by summation over the phases. Owing to the definition of the partial Hamiltonian $\Ham_k$ \eqref{eq:generic_partial_hamiltonian_def}, we have
\begin{equation}
m_k\Dt^k\pth{\frac{\Ham_k}{m_k}} = m_k\Dt^k\pth{\frac{\bm{K}_k^\top\vel_k}{m_k}} + \sum_{j=1}^{m}m_k\Dt^k\pth{\dfrac{\mathscr{M}_{k,j}\Dt^k\xi_j}{m_k}} - m_k\Dt^k\pth{\dfrac{\Lag_k}{m_k}}.
\end{equation}
We now derive a more suitable expression for $m_k\Dt^k\pth{\Lag_k/m_k} = \Dt^k\Lag_k - \Lag_k \Dt^k m_k/ m_k$. Using the chain rule, we have
\begin{multline}
\Dt^k\Lag_k = m_k\dfrac{\partial\Lag_k}{\partial m_k}\frac{\Dt^k m_k}{m_k} + \dfrac{\partial\Lag_k}{\partial\vel_k}\Dt^k\vel_k + \sum_{i=1}^{n_{k'}}\frac{\partial\Lag_k}{\partial b_{k,i}}\Dt^k b_{k,i} \\ + \sum_{j=1}^{m}\acco{\dfrac{\partial\Lag_k}{\partial\xi_j}\Dt^k\xi_j + \dfrac{\partial\Lag_k}{\partial\pth{\nabla\xi_j}}\Dt^k\pth{\nabla\xi_j} + \dfrac{\partial\Lag_k}{\partial\pth{\Dt^k\xi_j}}\Dt^k\pth{\Dt^k\xi_j}}
\end{multline}
Moreover, we have the following identities
\begin{align}
\dfrac{\partial\Lag_k}{\partial\vel_k}\Dt^k\vel_k &= m_k\Dt^k\pth{\frac{\bm{K}_k^\top\vel_k}{m_k}} - m_k\vel_k^\top\Dt^k\pth{\frac{\bm{K}_k}{m_k}},\\\dfrac{\partial\Lag_k}{\partial\pth{\Dt^k\xi_j}}\Dt^k\pth{\Dt^k\xi_j} &= m_k\Dt^k\pth{\frac{\mathscr{M}_{k,j}\Dt^k\xi_j}{m_k}} - m_k\pth{\Dt^k\xi_j}\Dt^k\pth{\frac{\mathscr{M}_{k,j}}{m_k}}.
\end{align}
Using these identities as well as the generic momentum equation for $\bm{K}_k$, it results that
\begin{multline}\label{A:DtkLag_k_big_expression}
m_k\Dt^k\pth{\frac{\Lag_k}{m_k}} = m_k\Dt^k\pth{\frac{\bm{K}_k^\top\vel_k}{m_k}} + \sum_{j=1}^{m}m_k\Dt^k\pth{\dfrac{\mathscr{M}_{k,j}\Dt^k\xi_j}{m_k}} \underbrace{- \Lag_k^\star\pth{\nabla\cdot\vel_k} - \vel_k^\top\nabla\Lag_k^\star}_{=T_1}\\
+\underbrace{\sum_{i=1}^{n_{k'}}\dfrac{\partial\Lag_k}{\partial b_{k',i}}\Dt^k b_{k',i} + \sum_{i=1}^{n_k}\frac{\partial\Lag_{k'}}{\partial b_{k,i}}\vel_k^\top\nabla b_{k,i} -  \sum_{i=1}^{n_{k'}}\frac{\partial\Lag_{k}}{\partial b_{k',i}}\vel_k^\top\nabla b_{k',i} }_{=T_2}\\
\sum_{j=1}^{m}\bigg\{
 \underbrace{\frac{\partial\Lag_k}{\partial\xi_j}\Dt^k\xi_j - \frac{\partial\Lag_k}{\partial\xi_j}\vel_k^\top\nabla\xi_j }_{=T_3 }
+ \underbrace{ m_k\vel_k^\top\nabla\xi_j \Dt^k\pth{\frac{\mathscr{M}_{k,j}}{m_k}}-m_k\Dt^k\xi_j\Dt^k\pth{\frac{\mathscr{M}_{k,j}}{m_k}} }_{=T_4} \\
 \underbrace{\frac{\partial\Lag_k}{\partial\pth{\nabla\xi_j}}\Dt^k\pth{\nabla\xi_j} - \vel_k^\top\nabla\cdot\pth{\nabla\xi_j\otimes\bm{D}_{k,j}} + \vel_k\top\pth{\nabla\cdot\bm{D}_{k,j}}\nabla\xi_j }_{=T_5}
\bigg\}.
\end{multline}
Here, for convenience, we used to notation $\Lag_k^\star = m_k\partial\Lag_k/\partial m_k - \Lag_k$. Terms $T_1$ to $T_5$ are then simplified into
\begin{subequations}
\begin{align}
T_1 &= - \nabla\cdot\pth{\Lag_k^\star\vel_k}, && T_2 = \sum_{i=1}^{n_{k'}}\frac{\partial\Lag_k}{\partial b_{k',i}}\partial_t b_{k',i} - \sum_{i=1}^{n_{k}}\frac{\partial\Lag_{k'}}{\partial b_{k,i}}\partial_t b_{k,i} ,\\
T_3 &= \sum_{j=1}^{m}\frac{\partial\Lag_k}{\partial\xi_j}\partial_t\xi_j, && T_4=-\sum_{j=1}^{m}m_k\Dt^k\pth{\frac{\mathscr{M}_{k,j}}{m_k}}\partial_t\xi_j,\\
T_5 &= \sum_{j=1}^{m}\acco{\nabla\cdot\pth{\bm{D}_{k,j}\partial_t\xi_j} - \nabla\cdot\bm{D}_{k,j}\partial_t\xi_j}.
\end{align}
\end{subequations}
Injecting these into \eqref{A:DtkLag_k_big_expression}, using the partial mass conservation and rearranging the various terms then yields
\begin{multline}
m_k\Dt^k\pth{\frac{\Ham_k}{m_k}} -\nabla\cdot\pth{\Lag_k^\star\vel_k + \sum_{j=1}^{m}\bm{D}_{k,j}\partial_t\xi_j} = \sum_{i=1}^{n_{k}}\frac{\partial\Lag_{k'}}{\partial b_{k,i}}\partial_t b_{k,i} - \sum_{i=1}^{n_{k'}}\frac{\partial\Lag_k}{\partial b_{k',i}}\partial_t b_{k',i} \\
+\sum_{j=1}^{m}\acco{\partial_t\mathscr{M}_{k,j} + \nabla\cdot\pth{\mathscr{M}_{k,j}\vel_k + \bm{D}_{k,j}} - \frac{\partial\Lag_k}{\partial\xi_j}}\partial_t\xi_j,
\end{multline}
which, owing to the partial mass equation, is equivalent to \eqref{eq:generic_partial_hamiltonian_eq}.

\section{Derivation of the new model in the  Eulerian framework}\label{app:derivation_eulerienne}
We derive the new all-topology model \eqref{syst:energy_form_with_lift} by Stationary Action Principle in the Eulerian framework. We consider the Lagrangian \eqref{eq:Lagrangian_Lin} and derive eqs. \eqref{sys:phasic_momentum_vol_frac_zeta} and \eqref{sys:mass_entropy_constraints}. This yields the isentropic version of the model, its energy form is then obtained through a straightforward change of variables to replace the phasic entropies $s_1$ and $s_2$ with the phasic total energy densities $m_1 E_1=m_1 e_1 + m_1\vel_1^2/2$ and $m_2 E_2=m_2 e_2 + m_2\vel_2^2/2$.

We consider the action defined through the Lagrangian \eqref{eq:Lagrangian_Lin}, and assume independent and compactly supported perturbations for each variable. Using integrations by parts, one may show that for any pair of $\mathscr{C}^1$ functions $\phi$, $\psi$, where at least one has compact support in $\{t\in(t_0,t_1),\, \bm{x}\in\Omega_t\}$, we have
\begin{equation}
\int_t\int_{\Omega_t} \phi\Dt^k\psi\mathrm{d}\bm{x}\mathrm{d}t = -\int_t\int_{\Omega_t} \psi\acco{\partial_t\phi + \nabla\cdot\pth{\phi \vel_k}}\mathrm{d}\bm{x}\mathrm{d}t,
\end{equation}
and an analogous result holds for material derivative and conservation at the mixture velocity $\vel$. Using this result, the critical points of the Action satisfy the following equations, each obtained by one of the independent perturbations as specified,
\begin{subequations}
\begin{align}
\delta\zeta: \qquad& \Dt\alpha_1=0,\label{B:vol_frac}\\
\delta\alpha_1: \qquad& \partial_t\pth{\rho\zeta} + \nabla\cdot\pth{\rho\zeta\vel} = p_1-p_2, \label{B:cons_zeta}\\
\delta\theta_k: \qquad& \Dt^k s_k =0,\\
\delta s_k: \qquad& \partial_t\pth{m_k\theta_k} + \nabla\cdot\pth{m_k\theta_k} = -m_k T_k, \label{B:cons_theta}\\
\delta \lambda_{k,i}: \qquad& \Dt^k\mu_{k,i}=0,\\
\delta \mu_{k,i}:  \qquad& \partial_t\pth{m_k\lambda_{k,i}} + \nabla\cdot\pth{m_k\lambda_{k,i}\vel_k},\label{B:cons_lambda}\\
\delta\xi_k: \qquad& \partial_t m_k + \nabla\cdot\pth{m_k\vel_k} =0,\label{B:mass_conservation}\\
\delta m_k: \qquad& \Dt^k\xi_k - \theta_k\Dt^k s_k - \zeta\Dt\alpha_1 - \sum_{i}\lambda_{k,i}\Dt^k\mu_{k,i} = \tfrac{1}{2}\vel_k^2 - h_k,\label{B:Dtkxik} \\
\delta \vel_k: \qquad & \vel_k = \nabla\xi_k - \theta_k\nabla s_k - \zeta\nabla\alpha_1 - \sum_i\lambda_{k,i}\nabla\mu_{k,i}.\label{B:vel_decomposition}
\end{align}
\end{subequations}
In particular, eq. \eqref{B:vol_frac} yields the volume fraction equation, while eq. \eqref{B:mass_conservation} yields the phasic mass conservation. Note that owing to the partial mass and total mass conservation laws, eqs. \eqref{B:cons_zeta}, \eqref{B:cons_theta} and \eqref{B:cons_lambda} may also be written
\begin{equation}\label{B:cons_to_Dt}
\Dt\zeta = (p_1-p_2)/\rho, \qquad \Dt^k\theta_k = -T_k, \qquad \Dt^k\lambda_{k,i}=0,
\end{equation}
and so we obtain the phasic entropy equations, and by applying the gradient operator, we obtain the evolution equation for $\bm{v}=\nabla\zeta$. At this stage, only the phasic momentum equations remain to be derived.  Using the following identity,
\begin{equation}
\Dt^k\pth{\phi\nabla\psi} = \nabla\pth{\phi\Dt^k\psi} + \pth{\nabla\psi}\Dt^k\phi - \pth{\nabla\phi}\Dt^k\psi + \pth{\nabla\vel_k}\pth{\phi\nabla\psi},
\end{equation}
we obtain the relation
\begin{multline}
\Dt^k\pth{\nabla\xi_k - \theta_k\nabla s_k - \zeta\nabla\alpha_1 - \sum_i\lambda_{k,i}\nabla\mu_{k,i}}  = \\ \nabla\brck{\Dt^k\xi_k - \theta_k\Dt^k s_k - \zeta\Dt^k\alpha_1 - \sum_i\lambda_{k,i}\Dt^k\mu_{k,i}} \\
+\pth{\nabla\theta_k}\Dt^k s_k - \pth{\nabla s_k}\Dt^k\theta_k + \pth{\nabla\zeta}\Dt^k\alpha_1 - \pth{\nabla\alpha_1}\Dt^k\zeta + \sum_i\acco{\pth{\nabla\lambda_{k,i}}\Dt^k\mu_{k,i} - \pth{\nabla\mu_{k,i}}\Dt^k\lambda_{k,i}} \\
-\pth{\nabla\vel_k}\pth{\nabla\xi_k - \theta_k\nabla s_k - \zeta\nabla\alpha_1 - \sum_i\lambda_{k,i}\nabla\mu_{k,i}}.
\end{multline}
Using the decomposition \eqref{B:vel_decomposition}, and the evolutions equations given by \eqref{B:Dtkxik} and \eqref{B:cons_to_Dt}, the previous equation writes
\begin{equation}
\Dt^k\vel_k + \pth{\nabla\vel_k}\vel_k = \nabla\pth{\tfrac{1}{2}\vel_k^2 - h_k} + T_k\nabla s_k + \pth{\nabla\zeta}\Dt^k\alpha_1 - \pth{\nabla\alpha_1}\Dt^k\zeta.
\end{equation}
We then use the thermodynamic identity $\mathrm{d}h_k = T_k\mathrm{d}s_k + \rho_k^{-1}\mathrm{d}p_k$, and the following equations,
\begin{align}
\Dt^k\alpha_1 &= \Dt\alpha_1 + \pth{\nabla\alpha_1}^\top (\vel_k-\vel) = \pth{\nabla\alpha_1}^\top (\vel_k-\vel),\\
\Dt^k\zeta &= \Dt\zeta + \pth{\nabla\zeta}^\top (\vel_k-\vel) = (p_1-p_2)/\rho + \pth{\nabla\zeta}^\top (\vel_k-\vel),
\end{align}
to obtain the velocity equation
\begin{equation}
\Dt^k\vel_k + \frac{1}{\rho_k}\nabla p_k + \frac{p_1-p_2}{\rho}\nabla\alpha_1 = \brck{\bm{v}\otimes\nabla\alpha_1 - \nabla\alpha_1\otimes\bm{v}}\pth{\vel_k-\vel},
\end{equation}
which, when multiplied by the partial mass $m_k$, yields the phasic momentum equation.

\vskip2pc

\bibliographystyle{RS.bst}

\end{document}